\documentclass[11pt]{amsart}

\usepackage{amssymb,amsmath,amssymb}
\usepackage{abstract}
\usepackage{amsthm}
\usepackage{mathrsfs}
\usepackage{graphicx}
\usepackage{subcaption}
\usepackage{float}
\usepackage{epstopdf}
\usepackage{makecell,rotating,multirow,diagbox}
\usepackage{booktabs}
\usepackage{makecell}
\usepackage{color}
\usepackage{url}
\usepackage[colorlinks,linkcolor=blue,anchorcolor=blue,citecolor=red]{hyperref}
\usepackage{appendix}
\usepackage{enumerate}
\usepackage{cite}
\usepackage{graphicx}
\usepackage{subcaption}

\def\bdiv{\mathrm{div}}

\newtheorem{theorem}{Theorem}[section]

\newtheorem{lemma}{Lemma}[section]
\newtheorem{remark}{Remark}[section]

\newtheorem{definition}{Definition}[section]

\numberwithin{equation}{section}

\DeclareMathOperator*{\argmin}{arg\,min}
\usepackage[svgnames]{xcolor}

\begin{document}

	\title[A mixed residual method for biharmonic equations]{A mixed residual method for biharmonic equations in spectral Barron spaces}
    \author{Mengjia Bai}
    \address{School of Mathematical Sciences, Fudan University, Shanghai 200433, China}
    \email{bai\_mj@fudan.edu.cn (M. Bai), 25110180015@m.fudan.edu.cn (K. Gai), slu@fudan.edu.cn (S. Lu)}
    \author{Kuo Gai}
    \author{Shuai Lu$^{\dag}$}
    \address{$^\dag$ Corresponding author, S. Lu.}
    
	\maketitle

	\begin{abstract}
We propose a mixed residual method (MIM) for numerically solving the biharmonic equation with nonhomogeneous clamped boundary conditions. By establishing the well-posedness of the biharmonic equation in spectral Barron spaces, we derive an error bound for MIM that relates shallow neural network approximations to the exact solution and overcomes the curse of dimensionality. This error bound consists of two components: the first corresponds to the approximation error of the neural network, while the second represents the generalization error arising from randomly sampled training data. Several numerical experiments are presented to demonstrate the effectiveness of the proposed method.

         
%
    
		\end{abstract}

\section{Introduction}
Partial differential equations (PDEs) play a central role in modeling a wide range of phenomena in the natural sciences. Over the past decades, numerous reliable and efficient numerical techniques have been developed for solving them. Classical discretization approaches, such as finite difference and finite element methods, are well established and widely used. However, applying these conventional methods to high-dimensional problems often encounters additional obstacles. The primary difficulty is that their computational cost increases exponentially with the dimension, which displays a challenge commonly known as the curse of dimensionality. This exponential growth renders approximating solutions to high-dimensional PDEs computationally prohibitive.

In recent years, neural network–based methodologies have increasingly demonstrated their value in numerical methods for PDEs, particularly in mitigating the curse of dimensionality \cite{han2018solving,raissi2019physics,yu2018deep,lyu2020enforcing,lyu2022mim,cai2020deep,zang2020weak,lee2025finite}. Among these, the Deep Galerkin Method (DGM) \cite{sirignano2018dgm} and Physics-Informed Neural Networks (PINNs) \cite{raissi2019physics} construct their training neural networks by minimizing the PDE residual in a least-squares setting. The Deep Ritz Method (DRM) \cite{yu2018deep}, in contrast, bases its loss function on the variational formulation of the PDE whenever such a formulation is available.
It should be noted that these above methods are primarily designed for low-order PDEs and may become computationally expensive for high-order PDEs, as evaluating derivatives of various orders is particularly costly in high dimensions. To address this issue, \cite{lyu2022mim,lyu2020enforcing} introduced the deep mixed residual method (MIM), which reformulates a high-order PDE into a system of lower-order equations and employs a single neural network to approximate both the solution and its derivatives. This reformulation reduces computational complexity and improves stability for high-order problems.

While neural network–based algorithms for solving PDEs have gained popularity and demonstrated empirical success, their theoretical analysis has consequently emerged as a key research focus. Further investigations are encouraged to extend the numerical analysis of such algorithms to a broader range of problems.  
Recent theoretical advances in this direction can be found in \cite{luo2024two,mishra2023estimates,shin2020convergence,shin2023error,xu2020finite,lu2021priori,bertoluzza2024wan} and the references therein. Most of these studies are situated within the Sobolev space framework and provide insights into the approximation properties and error bounds of neural network-based methods. It should be noted that some researchers have recently considered an alternative framework based on the spectral Barron space. For properties of such spaces and regularity results for elliptic PDEs, we refer to \cite{chen2023regularity,choulli2025functional,lumathe2025SIAP}. However, error estimates for PDEs within the spectral Barron space setting remain limited, and most existing analyses focus on second-order elliptic equations. For instance, \cite{lu2021priori,lu2022priori} analyzed the generalization error of the deep Ritz method (DRM) for the Poisson and Schrödinger equations, under the assumption that the PDE solutions belong to the spectral Barron space. \cite{li2024priori} derived a priori error estimates for the MIM based on the spectral Barron space for the Poisson equation. For higher-order elliptic equations with specific boundary conditions, we refer to \cite{bai2025error}, where a priori error estimates for the MIM within the spectral Barron space framework can be found.

In this work, we investigate the MIM for biharmonic equations with clamped boundary conditions in spectral Barron spaces. The biharmonic equation is a well-known higher-order PDE that arises in various fields of physics and applied mathematics, particularly in elasticity theory and Stokes flow. It plays a critical role in applications such as scattered data fitting with thin plate splines \cite{wahba1990spline}, fluid mechanics \cite{greengard1998integral,ferziger2002computational}, and linear elasticity \cite{christiansen1975integral,constanda1995boundary}. Based on recent regularity results for second-order elliptic equations, we provide regularity and error estimates for the boundary value problem. The key contribution of the present work is to extend such novel solution theory to biharmonic problems in high dimensions. While the framework of our error analysis can, in principle, be adapted to a broader class of higher-order elliptic equations, the associated boundary conditions are often considerably more intricate.

The remainder of this paper is organized as follows. In Section \ref{se2}, we present the basic setup and the main results of this work. The well-posedness theory in spectral Barron space is given in Section \ref{se3}. This section also introduces a general framework based on bilinear forms to outline the error decomposition. In Section \ref{se4}, we provide the proof of the main theorem. Finally, we conduct numerical experiments in Section \ref{se5} to support our theoretical results.

\section{Set-up and Main results}\label{se2}

In this section, we describe the set-up of the present work and state the main results, whose proofs can be found in the subsequent sections.

\subsection{Set-up of PDEs and MIM} 
Let $\Omega \subset (0,1)^d \subset \mathbb{R}^d$ be a domain with a sufficiently smooth boundary $\partial\Omega$.
We consider the following biharmonic equation equipped with the clamped boundary condition:
\begin{equation}\label{eq_biharmonicequation}
		\left\{
		\begin{aligned}
			\Delta^2 u = &\,f, \quad \mbox{in }\Omega,\\
			u = g,\ &	\partial_{\boldsymbol{n}} u = h, \quad \mbox{on }\partial\Omega,		
			\end{aligned}
			\right.
		\end{equation}
where $\boldsymbol{n}$ is the outward unit normal vector on the boundary $\partial\Omega$.

In this work, we employ the mixed formulation (MIM) to solve the biharmonic equation \eqref{eq_biharmonicequation} in order to mitigate the effects of high-order operators. Specifically, we introduce auxiliary shallow neural networks $\varphi_0$ and $\varphi_1$ to approximate $u$ and $\Delta u$, respectively. Combining the squared residual loss with a penalty term yields the mixed residual loss function:
\begin{equation}\label{eq_Lossfunctional}
\begin{aligned}
\mathcal{L}(\boldsymbol{u})
=&
    \|\Delta \varphi_{1} - f \|_{L^2(\Omega)}^2
    + 
    \|\Delta \varphi_0 -\varphi_{1}\|_{L^2(\Omega)}^2 \\
     &+
    \lambda\big(\|\varphi_0 - g \|_{L^2(\partial\Omega)}^2 
    +
    \|\partial_{\boldsymbol{n}}\varphi_0 - h \|_{L^2(\partial\Omega)}^2 \big), 
    \end{aligned}
\end{equation}
with $\boldsymbol{u} = (\varphi_0,\, \varphi_1)$. Here $\lambda > 0$ denotes a hyperparameter for which a specific choice is adopted in the analysis and numerical experiments. In addition, for simplicity and without affecting the proof of our main results, we assume that the penalties for both boundary conditions are equivalent. Meanwhile, we introduce the corresponding empirical loss function defined as
\begin{equation}\label{eq_EmpiricalLoss}
\begin{aligned}
	\mathcal{L}_n(\boldsymbol{u})
	=&
	\frac{|\Omega|}{n}\sum_{j=1}^{n} \left(\big|(\Delta\varphi_{1} - f) (X_j)\big|^2
    +
    \big|(\Delta\varphi_{0} - \varphi_1) (X_j)\big|^2\right)\\
	&+
	\frac{\lambda|\partial\Omega|}{\widehat{n}}\sum_{j=1}^{\widehat{n}} \left(
   \big|(\varphi_0 - g )(Y_j)\big|^2 + \big|(\partial_{\boldsymbol{n}}\varphi_0 - h)(Y_j)\big|^2 \right).
    \end{aligned}
\end{equation}
Here, $\{X_j\}_{j=1}^{n} $ and $\{Y_j\}_{j=1}^{\widehat{n}} $ are randomly sampled from $\Omega$  and $\partial\Omega$, respectively. 

Given an empirical loss function $\mathcal{L}_n$, the empirical loss minimization algorithm seeks $\boldsymbol{u}_{\theta^*}$ that minimizes $\mathcal{L}_n$, i.e., 
\begin{equation*}
    \boldsymbol{u}_{\theta^*} = \argmin_{\boldsymbol{u} \in \mathcal{F}_{\Theta}(B)} \mathcal{L}_n,
\end{equation*}
where $\mathcal{F}_{\Theta}(B)$ is a hypothesis class of functions parameterized by neural networks to be introduced later.

\subsection{Neural networks} 
A neural network is a nonlinear parametric model defined as a concatenation of affine maps and activation functions. 
Let $\boldsymbol{x}\in \mathbb{R}^d$ be the input element and $\boldsymbol{y}\in \mathbb{R}^{m}$ be the output vector function. An $L$-layer neural network can be written as:
\begin{equation*}
\begin{aligned}
\boldsymbol{h}_0& = \boldsymbol{x},\\
\boldsymbol{h}_\ell
	&=
\sigma(
W_\ell \boldsymbol{h}_{\ell-1} + b_\ell
)
	\quad \mbox{for} \quad
	\ell=1,\cdots, L-1,\\
 \boldsymbol{y}
	&=
W_L \boldsymbol{h}_{L-1} + b_L.
 \end{aligned}
	\end{equation*}
Especially, a neural network is classified as shallow if its depth $L=2$, and as deep otherwise.
Here, $W_\ell$ denotes the weight matrix of the $\ell$-th layer, and $\boldsymbol{b}_\ell$ represents the bias vector. The function $\sigma$ is referred to as an activation function, which introduces non-linearity to enhance the model's expressive power.

In this work, we focus on shallow neural networks by fixing the depth to $L = 2$. In particular, we choose the activation function $\sigma = \sigma_k := \mathrm{ReLU}^k$, defined as follows:
\begin{definition}\label{def: relu}
    Let $k\in \mathbb{N}$, the nonlinear univariate function $\mathrm{ReLU}^k(x)$ is defined by 
    \begin{equation*}
    \mathrm{ReLU}^k(x)
    =
    \left\{
        \begin{aligned}
            &x^k,
            \quad 
            &x\geq 0,\\
            &0,\quad
            &x<0.
        \end{aligned}
        \right.
    \end{equation*}
     In particular, we denote $\mathrm{ReQU}(x) =\mathrm{ReLU}^2(x)$ and $\mathrm{ReCU}(x) =\mathrm{ReLU}^3(x)$. 
\end{definition}

\subsection{Main results}
The goal of this paper is to obtain quantitative estimates for the error between the minimal solution $\boldsymbol{u}_{\theta^*}$, which is a set of shallow neural networks computed from the finite data points, and the exact solutions of the biharmonic equation \eqref{eq_biharmonicequation}. Our primary interest is to derive such estimates that scale mildly with respect to the increasing dimension $d$. To this end, it is necessary to assume that the true solutions lie in a hypothesis function space of a certain complexity. Specifically, referring to \cite{choulli2025functional}, we consider the spectral Barron space defined below. 

For the sake of completeness, we recall some necessary definitions of spectral Barron spaces from \cite{choulli2025functional}.
Concerning whole space $\mathbb{R}^d$, we define the spectral Barron space $\mathcal{B}^s(\mathbb{R}^d)$ for $s\geq 0$ as
\begin{equation*}
	\mathcal{B}^s(\mathbb{R}^d) =
	 \{f\in C_b^0(\mathbb{R}^d); \langle\boldsymbol{\xi}\rangle^s\hat{f}\in L^1(\mathbb{R}^d)\},
	\end{equation*}
with the natural norm
\begin{equation*}
	\|f\|_{\mathcal{B}^s(\mathbb{R}^d)}
	:=
	\|\langle\boldsymbol{\xi}\rangle^s\hat{f}\|_{L^1(\mathbb{R}^d)},\quad
	f\in \mathcal{B}^s(\mathbb{R}^d),
	\end{equation*}
where $\langle\boldsymbol{\xi}\rangle := \sqrt{1+|\boldsymbol{\xi}|^2} $, and $C_b^0$ is the space of bounded continuous functions, i.e. $C_b^0 = L^\infty \cap C^0$. This definition is based on the symbol of $I-\Delta$; we omit further details on the properties of the spectral Barron space and refer the reader to \cite{choulli2025functional} for a more comprehensive discussion.

To further define a spectral Barron space on a bounded domain $\Omega$, for $s\geq 0$, we define the closed subspace $\mathcal{F}_s$ of $\mathcal{B}^s(\mathbb{R}^d)$ as follows:
\begin{equation*}
	\mathcal{F}_s
	:=
	\{f\in \mathcal{B}^s(\mathbb{R}^d);\ \mbox{sup}(f)\subset \mathbb{R}^d\setminus \Omega\}.
	\end{equation*} 
Then we introduce the quotients space $\mathcal{B}^s(\Omega): =\mathcal{B}^s(\mathbb{R}^d)/\mathcal{F}_s$. In the sequel, $\pi_s: \mathcal{B}^s(\mathbb{R}^d)\rightarrow \mathcal{B}^s(\Omega)$ denotes the standard quotient map. For any $g_1, g_2\in \mathcal{B}^s(\mathbb{R}^d)$, the equality $\pi_s(g_1)=\pi_s(g_2)$ implies that $g_1-g_2\in \mathcal{F}_s$, or equivalently $g_1|_{\bar{\Omega}} = g_2|_{\bar{\Omega}}$. We endow $\mathcal{B}^s(\Omega)$ with the quotient norm defined by
\begin{equation}\label{eq_boundedbarronnorm}
 \|f\|_{\mathcal{B}^s(\Omega)}
	=
	\inf\{g\in \mathcal{B}^s(\mathbb{R}^d);\ \pi_s{g}=f\},\quad f\in \mathcal{B}^s(\Omega).
	\end{equation}
Consequently, $\mathcal{B}^s(\Omega)$ is a Banach space with respect to the norm $\|\cdot\|_{\mathcal{B}^s(\Omega)}$.

Furthermore, since the biharmonic equation considered in \eqref{eq_biharmonicequation} has a nonhomogeneous boundary condition, we set $\Gamma = \partial\Omega$ and take $s > 0$. We define the trace space 
\begin{equation*}
	\mathcal{B}^s(\Gamma)
	:=
	\{g=u|_{\Gamma},;\ u\in \mathcal{B}^s(\Omega)\},
	\end{equation*}
equipped with the quotient norm 
\begin{equation*}
	\|g\|_{\mathcal{B}^s(\Gamma)}
	:=
	\inf\{\|u\|_{\mathcal{B}^s(\Omega)}; u|_\Gamma=g\}.
	\end{equation*}

The most important property of spectral Barron space is that functions in these spaces can be well approximated by shallow neural networks without the curse of dimensionality, as shown, for instance, in \cite{KluBarron2018IEEE, lumathe2025SIAP}. To make this more precise, let us define the class of shallow neural networks to be used for solving PDEs. Given a constant $B > 0$ and a number of neurons $m$, we define the hypothesis class
\begin{equation}\label{eq_networkfunctionclass}
\mathcal{F}_{m}(B)= \left\{ \sum_{i=1}^m \gamma_i \sigma_3(\boldsymbol{\omega}_i\cdot \boldsymbol{x} + t_i)\in \mathbb{R} : |\boldsymbol{\omega}_i|_1 =1, t_i\in [-1,1), \sum_{i=1}^m|\gamma_i|\leq B \right\}, 
\end{equation}
where $\sigma_3 = \mathrm{ReCU}$ is the activation function in Definition \ref{def: relu}. 

Our first result establishes the well-posedness of the biharmonic equation \eqref{eq_biharmonicequation} in spectral Barron spaces, focusing on the existence of solutions.
\begin{theorem}\label{thm1}
	Let $ s\geq 0,\ k>s+d/2$. For all $(f,g,h)\in H^k(\Omega)\times H^{k+7/2}(\partial\Omega)\times H^{k+5/2}(\partial\Omega)$, then the BVP \eqref{eq_biharmonicequation} admits a unique solution $u\in \mathcal{B}^{4+s}(\Omega)$. Furthermore, 
     the following inequality holds:
	\begin{equation*}
		\|u\|_{\mathcal{B}^{4+s}(\Omega)}
		\leq
		C(\|f\|_{H^k(\Omega)}
        +\|g\|_{ H^{k+7/2}(\partial\Omega)}
        +\|h\|_{H^{k+5/2}(\partial\Omega)}),
		\end{equation*}
        where $C > 0$ is a constant.
	\end{theorem}
We present the proof in Section \ref{subse_Wellposedness}.

Based on the above existence results, we observe that the solution of \eqref{eq_biharmonicequation} lies in a spectral Barron space. Furthermore, the quantitative estimates for the error between the exact solution and the minimal solution $\boldsymbol{u}_{\theta^*}$ are given as follows.
\begin{theorem}\label{thm2}
Let $u^*\in \mathcal{B}^6(\Omega)$ be the solution of biharmonic problem \eqref{eq_biharmonicequation}, and $K=\max\{\|u^*\|_{\mathcal{B}^6(\Omega)}, \|f\|_{L^\infty(\Omega)}, \|g\|_{L^\infty(\partial\Omega)}, \|h\|_{L^\infty(\partial\Omega)} \}$. Denote $$\boldsymbol{u}_{\theta^*}
=
(\varphi_0^{*},\, \varphi_1^*)
=
\argmin_{\varphi_0,\varphi_1\in \mathcal{F}_{m}(c\|u^*\|_{\mathcal{B}^6(\Omega)})}
\mathcal{L}_n(\varphi_0,\varphi_1)$$
with a constant $c>0$. Then there holds
    \begin{equation*}
		\begin{aligned}
			\mathbb{E}\|\varphi_0^* - u^*\|_{H^1(\Omega)}^2
            +
			\mathbb{E}\|\varphi_1^* - \Delta u^*\|_{L^2(\Omega)}^2
            +
            \mathbb{E}\|\Delta\varphi_1^* - \Delta^{2} u^*\|_{L^2(\Omega)}^2
			\\
			\leq
			C\|u^*\|_{\mathcal{B}^{6}(\Omega)} m^{-(\frac{1}{2}+\frac{1}{3d})}
			+
			K n^{-\frac{1}{4}},
		\end{aligned}
	\end{equation*}
	where $C>0$ is a constant, and the expectation is taken on the random sampling of training data in $\Omega$ and $\partial\Omega$.
\end{theorem}

The error bound in Theorem \ref{thm2} consists of two components: the first corresponds to the approximation error of the neural network, while the second represents the generalization error arising from the use of randomly sampled training data. Moreover, this error bound demonstrates that the convergence rates of the neural network solution to the biharmonic equation do not suffer from the curse of dimensionality, provided that the exact solution belongs to the spectral Barron space. A detailed proof is deferred to Section \ref{se4}.

\section{Solution existence and error decomposition}\label{se3}
In this section, we establish the existence of a solution to the biharmonic equation within the spectral Barron space, drawing inspiration from the well-posedness framework introduced in \cite{choulli2025functional}. Furthermore, we show how the error between the exact solution and the approximate solution can be controlled by the approximation error and the generalization error for the biharmonic problem.

\subsection{Well-posedness}\label{subse_Wellposedness}
Before proceeding with the proof of Theorem \ref{thm1}, we define, for $0\leq s<t$ and $\varrho>0$, the set $\mathcal{D}_\varrho^{s,t}$ as the closure of the set  
\begin{equation*}
	\mathcal{D}_\varrho^{t}
	:=
	\{(f,g,h)=(\Delta^2 u, u|_\Gamma,\partial_{\boldsymbol{n}} u|_\Gamma);\; u\in \mathcal{B}^{4+t}(\Omega),\; \|u\|_{\mathcal{B}^{4+t}(\Omega)}\leq \varrho\}
\end{equation*}
 in $\mathcal{B}^s(\Omega)\times\mathcal{B}^{s+4}(\Gamma)\times\mathcal{B}^{s+3}(\Gamma)$. Such a set allows for additional regularity of the source term and boundary conditions, which yields the following existence and uniqueness result.
\begin{lemma}\label{lem_existence}
 Let $0 \leq s < t$, $\varrho > 0$, and $(f,g,h) \in \mathcal{D}_\varrho^{s,t}$. Then there exists a unique solution $u \in \mathcal{B}^{4+s}(\Omega)$ of the boundary value problem for the biharmonic equation \eqref{eq_biharmonicequation} such that $\|u\|_{\mathcal{B}^{4+s}(\Omega)} \leq \varrho$.
	\end{lemma}

\begin{proof}
We first prove existence. Let  
 $$ \{(f_j, g_j, h_j)\} = \{(\Delta^2 u_j,\; u_j|_{\Gamma},\; \partial_{\boldsymbol{n}} u_j|_{\Gamma})\} $$  
 be a sequence in $\mathcal{D}_\varrho^{t}$ converging to $(f,g,h)$ in $\mathcal{B}^s(\Omega) \times \mathcal{B}^{s+4}(\Gamma) \times \mathcal{B}^{s+3}(\Gamma)$, where $\{u_j\}$ is a sequence in $\mathcal{B}^{4+t}(\Omega)$ satisfying $\sup_j \|u_j\|_{\mathcal{B}^{4+t}(\Omega)} \le \varrho$.  
Since the embedding $\mathcal{B}^{4+t}(\Omega) \hookrightarrow \mathcal{B}^{4+s}(\Omega)$ is compact by \cite[Prop. 4.1]{choulli2025functional}, after extracting a subsequence if necessary we may assume that $\{u_j\}$ converges to some $u \in \mathcal{B}^{4+s}(\Omega)$ in $\mathcal{B}^{4+s}(\Omega)$.  Consequently, $\|u\|_{\mathcal{B}^{4+s}(\Omega)} \le \varrho$ and  
$$ \{(\Delta^2 u_j,\; u_j|_{\Gamma},\; \partial_{\boldsymbol{n}} u_j|_{\Gamma})\} \to \{(\Delta^2 u,\; u|_{\Gamma},\; \partial_{\boldsymbol{n}} u|_{\Gamma})\} $$  
in $\mathcal{B}^s(\Omega) \times \mathcal{B}^{s+4}(\Gamma) \times \mathcal{B}^{s+3}(\Gamma)$. By uniqueness of the limit we obtain $$\{(\Delta^2 u,\; u|_{\Gamma},\; \partial_{\boldsymbol{n}} u|_{\Gamma})\} = (f,g,h). $$
	\end{proof}

We further establish the regularity of the solution, provided that the source term and the boundary condition are sufficiently smooth.

\begin{lemma}\label{lem_regularity}
	Let $s\geq 0$, and $k>s+d/2$ be integers. For all $(f,g,h)\in H^k(\Omega)\times H^{k+7/2}(\Gamma)\times H^{k+5/2}(\Gamma)$, the BVP \eqref{eq_biharmonicequation} admits a unique solution $u\in \mathcal{B}^{4+s}(\Omega)$, and the following stability estimate holds:
	\begin{equation*}
		\|u\|_{\mathcal{B}^{4+s}(\Omega)}
		\leq
		C(\|f\|_{H^k(\Omega)}+\|g\|_{ H^{k+7/2}(\Gamma)}+\|h\|_{H^{k+5/2}(\Gamma)}).
		\end{equation*}
	\end{lemma}
\begin{proof}
	According to the standard elliptic estimates, we deduce from \cite[Thm 2.16]{gazzola2010polyharmonic} that $u\in H^{4+k}(\Omega)$ and
	\begin{equation*}
		\|u\|_{H^{4+k}(\Omega)}
		\leq
		C(\|f\|_{H^k(\Omega)}+\|g\|_{ H^{k+7/2}(\Gamma)}+\|h\|_{H^{k+5/2}(\Gamma)}).
		\end{equation*}

On the other hand, from \cite{grisvard2011elliptic}, there exists $v\in H^{4+k}(\mathbb{R}^d)$ such that $v|_\Omega=u$ and
\begin{equation*}
	\|v\|_{H^{4+k}(\mathbb{R}^d)}
	\leq
	C	\|u\|_{H^{4+k}(\Omega)}.
	\end{equation*}
Hence we have
\begin{equation*}
	\|v\|_{H^{4+k}(\mathbb{R}^d)}
	\leq
		C(\|f\|_{H^k(\Omega)}+\|g\|_{ H^{k+7/2}(\Gamma)}+\|h\|_{H^{k+5/2}(\Gamma)}).
	\end{equation*}
Take $s<t<k- d/2$. From \cite[Prop. 4.1]{choulli2025functional}, we know that $H^{4+k}(\mathbb{R}^d)\hookrightarrow \mathcal{B}^{4+t}(\mathbb{R}^d)$. Thus $v\in \mathcal{B}^{4+t}(\mathbb{R}^d)$, it yields that
\begin{equation*}
	\|v\|_{\mathcal{B}^{4+t}(\mathbb{R}^d)}
	\leq
		C(\|f\|_{H^k(\Omega)}+\|g\|_{ H^{k+7/2}(\Gamma)}+\|h\|_{H^{k+5/2}(\Gamma)})=:\varrho.
	\end{equation*}
The proof is then completed by applying Lemma \ref{lem_existence}.
	\end{proof}

 \begin{proof}[Proof of Theorem \ref{thm1}]
 By combining Lemma \ref{lem_existence} and Lemma \ref{lem_regularity}, we directly obtain the argument for the theorem.
    \end{proof}

\subsection{Error Decomposition}
To derive the error bound for MIM, the central idea of error decomposition is to follow the general approach from \cite{zeinhofer2025unified}, where one applies Céa's lemma after establishing the boundedness and coercivity of the bilinear form induced by \eqref{eq_Lossfunctional}.
  
For simplicity, we first introduce an abstract framework. Let $X$ and $Y$ be two Hilbert spaces, and let a linear map
\begin{equation*}
	T: X \rightarrow Y,\quad \boldsymbol{u} \mapsto T\boldsymbol{u},
\end{equation*}
be given. Define the corresponding bilinear form for any $\boldsymbol{u}, \boldsymbol{v} \in X$:
\begin{equation}\label{eq_generalbilinear}
    \mathfrak{B}(\boldsymbol{u},\boldsymbol{v}) = (T\boldsymbol{u}, T\boldsymbol{v})_Y.
\end{equation}
Furthermore, let $\mathcal{F}_\Theta$ denote the collection of all neural network functions with a given parameter space $\Theta$. Building upon C\'ea's lemma from \cite{zeinhofer2025unified}, we obtain the following extended version:
\begin{lemma}\cite[Generalized C\'ea's Lemma]{zeinhofer2025unified}\label{lem_cea}
For any function $\boldsymbol{f} \in Y$, define the loss function $\mathcal{L}(\boldsymbol{u}) = \|T\boldsymbol{u} - \boldsymbol{f}\|_Y^2$, and let $\boldsymbol{u}^*$ denote its unique minimizer.  
Assume that the bilinear form $\mathfrak{B}$ defined in \eqref{eq_generalbilinear} is bounded, i.e.,
\[
\mathfrak{B}(\boldsymbol{u},\boldsymbol{v}) \precsim \|\boldsymbol{u}\|_X \|\boldsymbol{v}\|_X.
\]
Moreover, assume that there exists a Hilbert space $Z$ with $X \subseteq Z$ such that $\mathfrak{B}$ is sup-linearly coercive with respect to $Z$, i.e., $\mathfrak{B}(\boldsymbol{u},\boldsymbol{u}) \ge a \|\boldsymbol{u}\|_Z^4$. Then the following estimate holds:
\begin{align}\label{eq_totalerror}
\|\boldsymbol{u}_{\theta} - \boldsymbol{u}^*\|_Z^2 \precsim \sqrt{\delta(\boldsymbol{u}_\theta)} + \sqrt{\inf_{\hat{\boldsymbol{u}}_{\theta} \in \mathcal{F}_\Theta} \|\hat{\boldsymbol{u}}_{\theta} - \boldsymbol{u}^*\|_X^2},
\end{align}
where $\delta(\boldsymbol{u}_{\theta}) = \mathcal{L}(\boldsymbol{u}_{\theta}) - \inf_{\hat{\boldsymbol{u}}_{\theta} \in \mathcal{F}_\Theta} \mathcal{L}(\hat{\boldsymbol{u}}_{\theta}).$
\end{lemma}

In the above total error bound (\ref{eq_totalerror}), the second term $\sqrt{\inf_{\hat{\boldsymbol{u}}_{\theta} \in \mathcal{F}_\Theta} \|\hat{\boldsymbol{u}}_{\theta} - \boldsymbol{u}^*\|_X^2}$ is referred to as the approximation error. We further decompose the first term $\sqrt{\delta(\boldsymbol{u}_\theta)}$ below. 

\begin{remark}\label{rem_decompose}
Let $\mathcal{L}_n(\boldsymbol{u})$ be the empirical loss function corresponding to $\mathcal{L}(\boldsymbol{u})$. We further decompose the error term $\delta(\boldsymbol{u}_\theta)$ as follows:
    \begin{equation*}
    \begin{aligned}
       \delta(\boldsymbol{u}_\theta)
       =&
       \mathcal{L}(\boldsymbol{u}_{\theta})
       -
       \mathcal{L}_n(\boldsymbol{u}_{\theta})
       +
      \mathcal{L}_n(\boldsymbol{u}_{\theta})
       -
       \inf_{\hat{\boldsymbol{u}}_{\theta}\in \mathcal{F}_\Theta}
       \mathcal{L}_n(\hat{\boldsymbol{u}}_{\theta})
       +
       \inf_{\hat{\boldsymbol{u}}_{\theta}\in \mathcal{F}_\Theta}
       \mathcal{L}_n(\hat{\boldsymbol{u}}_{\theta})
       -
       \inf_{\hat{\boldsymbol{u}}_{\theta}\in \mathcal{F}_\Theta}\mathcal{L}(\hat{\boldsymbol{u}}_{\theta})\\
       \lesssim&
       \underbrace{\sup_{\hat{\boldsymbol{u}}_{\theta}\in \mathcal{F}_\Theta}
       |\mathcal{L}(\hat{\boldsymbol{u}}_{\theta})
       -
       \mathcal{L}_n(\hat{\boldsymbol{u}}_{\theta})|}_{\mbox{generalization error}}
       +
        \underbrace{\big(\mathcal{L}_n(\boldsymbol{u}_{\theta})
       -
       \inf_{\hat{\boldsymbol{u}}_{\theta}\in \mathcal{F}_\Theta}
       \mathcal{L}_n(\hat{\boldsymbol{u}}_{\theta})\big)}_{\mbox{optimization error}}.
    \end{aligned}
    \end{equation*}
In the above inequality, the first term is called the "global generalization error", and the second term is the "optimization error". Combining this with Lemma~\ref{lem_cea}, the total error bound is of the form "global generalization error + optimization error + approximation error".
\end{remark}

\subsection{Coercivity and boundedness}
To apply the Céa's lemma, we verify boundedness and coercivity of the bilinear form $\mathfrak{B}(\cdot,\cdot)$. To this end, we consider two functions \(\boldsymbol{u}, \boldsymbol{w}\in H^2(\Omega)\times H^2(\Omega)\) defined by
\begin{equation}\label{form of U and V}
	\begin{aligned}
		\boldsymbol{u} =(\varphi_0,\,\varphi_{1}),\
		\boldsymbol{w} =(\psi_0,\,\psi_{1}),
	\end{aligned}
\end{equation}
where \(\varphi_i,\,\psi_i \in \mathbb{R}\) for \(i = 0,\,1\). Moreover, we define the following bilinear form, induced from \eqref{eq_Lossfunctional}:
\begin{equation}\label{bilinear form}
    \begin{aligned}
        \mathfrak{B}(\boldsymbol{u},\boldsymbol{w})
        =&
        (\Delta \varphi_{1}, \Delta\, \psi_{1})
        +
        (\Delta \varphi_0 -\varphi_{1}, \Delta \psi_0 -\psi_{1})
        \\
   & + 
    \lambda\big((\varphi_0, \psi_0)_{\partial\Omega}   
    +
    (\partial_{\boldsymbol{n}}\varphi_0,  \partial_{\boldsymbol{n}}\psi_0)_{\partial\Omega} \big),
    \end{aligned}
\end{equation}
where \((\cdot,\cdot)_{\partial\Omega}\) denotes the inner product defined on \(\partial\Omega\). The boundness and coercivity results of the bilinear form $\mathfrak{B}(\cdot,\cdot)$ are established below. 

\begin{lemma}\label{lem_coercivity}
For any functions \(\boldsymbol{u},\,\boldsymbol{w}\in H^2(\Omega)\times H^2(\Omega)\) as defined in \eqref{form of U and V}, the bilinear form \eqref{bilinear form} is bounded; i.e.,
\begin{equation*}
	\begin{aligned}
		\mathfrak{B}(\boldsymbol{u},\boldsymbol{w})
		\le 
		C_B\|\boldsymbol{u}\|_{H^2(\Omega)}
		\|\boldsymbol{w}\|_{H^2(\Omega)}.
	\end{aligned}
\end{equation*}
Moreover, if \(\| \boldsymbol{u} \|_{H^2(\Omega)}\le E\) holds true, then the following sup-linear coercivity holds:
\begin{equation*}
	\begin{aligned}
		\sqrt{\mathfrak{B}(\boldsymbol{u},\boldsymbol{u})}
		\geq
		C_E
		\Big( \|\varphi_0\|_{H^1(\Omega)}^2 
		+ \|\Delta \varphi_0\|_{L^2(\Omega)}^2 
		+
		\|\varphi_1\|_{L^2(\Omega)}^2 + \|\Delta\varphi_1\|_{L^2(\Omega)}^2\Big).
	\end{aligned}
\end{equation*}
Here, \(C_B\) and \(C_E\) are positive constants.
\end{lemma}

\begin{proof}
By repeatedly applying the Cauchy–Schwarz and trace inequalities, we obtain
\begin{equation*}
\begin{aligned}
&|\mathfrak{B}(\boldsymbol{u},\boldsymbol{w})|
\;\le\;
\|\Delta \varphi_1\|_{L^2(\Omega)} 
\|\Delta \psi_1\|_{L^2(\Omega)}
\\[4pt]
&+
2\Big(
\|\Delta \varphi_0\|_{L^2(\Omega)}^2 
+ \|\varphi_1\|_{L^2(\Omega)}^2
\Big)^{1/2}
\Big(
\|\Delta \psi_0\|_{L^2(\Omega)}^2 
+ \|\psi_1\|_{L^2(\Omega)}^2
\Big)^{1/2}
\\[4pt]
&+
\lambda
\Big(
\|\varphi_1\|_{L^2(\partial\Omega)} 
\|\psi_1\|_{L^2(\partial\Omega)}
+
\|\partial_{\boldsymbol{n}}\varphi_1\|_{L^2(\partial\Omega)}
\|\partial_{\boldsymbol{n}}\psi_1\|_{L^2(\partial\Omega)}
\Big)
\\[6pt]
\le\;&
\Big(
\|\Delta \varphi_1\|_{L^2(\Omega)}^2
+
2(\|\Delta \varphi_0\|_{L^2(\Omega)}^2 
+\|\varphi_1\|_{L^2(\Omega)}^2)
+
\lambda^2(
\|\varphi_1\|_{L^2(\partial\Omega)}^2
+
\|\partial_{\boldsymbol{n}}\varphi_1\|_{L^2(\partial\Omega)}^2)
\Big)^{1/2}
\\
&\times
\Big(
\|\Delta \psi_1\|_{L^2(\Omega)}^2
+
2(\|\Delta \psi_0\|_{L^2(\Omega)}^2 
+\|\psi_1\|_{L^2(\Omega)}^2)
+
\lambda^2(
\|\psi_1\|_{L^2(\partial\Omega)}^2
+
\|\partial_{\boldsymbol{n}}\psi_1\|_{L^2(\partial\Omega)}^2)
\Big)^{1/2}
\\[6pt]
\le\;&
C_B(\|\varphi_1\|_{H^2(\Omega)}^2
+
\|\varphi_0\|_{H^2(\Omega)}^2)^{1/2}
(\|\psi_1\|_{H^2(\Omega)}^2
+
\|\psi_0\|_{H^2(\Omega)}^2)^{1/2},
\end{aligned}
\end{equation*}
which yields the boundedness.

For the sup-linear coercivity, applying Green's formula yields
\begin{equation*}
    (\Delta \varphi_{0}, \varphi_{1})
    =
    \big(\varphi_{0}, \Delta \varphi_{1}\big)	
		+
		\big(\partial_{\boldsymbol{n}} \varphi_{0}, \varphi_{1} \big)_{\partial\Omega}
		-
		\big( \varphi_{0},\partial_{\boldsymbol{n}} \varphi_{1} \big)_{\partial\Omega}.
\end{equation*}
Hence, for \(0 < \delta < 1\), it follows that
\begin{equation*}
    \begin{aligned}
        \mathfrak{B}(\boldsymbol{u},\boldsymbol{u})
        \geq\,&
        \|\Delta \varphi_1\|_{L^2(\Omega)}^2
        +
        \delta \|\Delta \varphi_0 - \varphi_1\|_{L^2(\Omega)}^2
        +
        \lambda\big(\|\varphi_0\|_{L^2(\partial\Omega)}^2
        +
        \|\partial_{\boldsymbol{n}}\varphi_0\|_{L^2(\partial\Omega)}^2\big)\\
        =\,&
        \|\Delta \varphi_1\|_{L^2(\Omega)}^2 + 
        \delta \big(\|\Delta \varphi_0\|_{L^2(\Omega)}^2 + \|\varphi_1\|_{L^2(\Omega)}^2 \big)
         + 2\delta( \varphi_{0}, \Delta \varphi_{1}) + \mathcal{T}_{\mathrm{b}}, 
    \end{aligned}
\end{equation*}
where 
\begin{equation*}
    \begin{aligned}
        \mathcal{T}_{\mathrm{b}}
        :=
        \lambda\big(\|\varphi_0\|_{L^2(\partial\Omega)}^2
        +
    \|\partial_{\boldsymbol{n}}\varphi_0\|_{L^2(\partial\Omega)}^2\big)
    +
        \delta\big(\partial_{\boldsymbol{n}} \varphi_{0}, \varphi_{1} \big)_{\partial\Omega}
		-
		\delta\big( \varphi_{0},\partial_{\boldsymbol{n}} \varphi_{1} \big)_{\partial\Omega}.
    \end{aligned}
\end{equation*}
For the cross term, we know
\begin{equation*}
2\delta( \varphi_{0}, \Delta\varphi_{1}) 
=
\frac{1}{2}\|\Delta\varphi_{1} + 2\delta\varphi_{0}\|_{L^2(\Omega)}^2
- 
\frac{1}{2}\|\Delta\varphi_{1}\|_{L^2(\Omega)}^2
-
2\delta^2 \|\varphi_{0}\|_{L^2(\Omega)}^2. 
\end{equation*}
Hence we have
\begin{equation*}
    \begin{aligned}
        \mathfrak{B}(\boldsymbol{u},\boldsymbol{u})
        \geq\,
        \frac{1}{2}\|\Delta\varphi_{1}\|_{L^2(\Omega)}^2 
        +
        \delta \big(\|\Delta \varphi_0\|_{L^2(\Omega)}^2 + \|\varphi_1\|_{L^2(\Omega)}^2 \big) 
        - 
        2\delta^2 \|\varphi_{0}\|_{L^2(\Omega)}^2
        +
        \mathcal{T}_{\mathrm{b}}.
    \end{aligned}
\end{equation*}
In addition, applying H\"older's inequality yields
    \begin{equation*}
    \begin{aligned}
        \big(\partial_{\boldsymbol{n}} \varphi_{0}, \varphi_{1} \big)_{\partial\Omega}
		-&
		\big( \varphi_{0},\partial_{\boldsymbol{n}} \varphi_{1} \big)_{\partial\Omega}\\
        \geq\, &
        -\|\partial_{\boldsymbol{n}}\varphi_{0}\|_{L^2(\partial\Omega)}
        \|\varphi_{1}\|_{L^2(\partial\Omega)}
        -
        \|\varphi_{0}\|_{L^2(\partial\Omega)}
        \|\partial_{\boldsymbol{n}}\varphi_{1}\|_{L^2(\partial\Omega)}\\
        \geq\, &
        -E (\|\partial_{\boldsymbol{n}}\varphi_{0}\|_{L^2(\partial\Omega)}+  \|\varphi_{0}\|_{L^2(\partial\Omega)})
        \geq\, -E \sqrt{\mathfrak{B}(\boldsymbol{u},\boldsymbol{u})} .
        \end{aligned}
    \end{equation*}
Here, we use the assumption $\|\boldsymbol{u}\|_{H^2(\Omega)}\leq\, E$.
Therefore, we conclude that
\begin{equation*}
\begin{aligned}
    \mathfrak{B}(\boldsymbol{u},\boldsymbol{u}) + E \sqrt{\mathfrak{B}(\boldsymbol{u},\boldsymbol{u})}
    \geq\,&
    \frac{1}{2}\|\Delta\varphi_{1}\|_{L^2(\Omega)}^2 
        +
        \delta \big(\|\Delta \varphi_0\|_{L^2(\Omega)}^2 + \|\varphi_1\|_{L^2(\Omega)}^2 \big) \\
        &- 
        2\delta^2 \|\varphi_{0}\|_{L^2(\Omega)}^2
        +
        \lambda\big(\|\varphi_0\|_{L^2(\partial\Omega)}^2
        +
        \|\partial_{\boldsymbol{n}}\varphi_0\|_{L^2(\partial\Omega)}^2\big).
        \end{aligned}
\end{equation*}
Moreover, we apply the standard elliptic estimate for the Neumann boundary condition, see \cite[Prop. 2.10 ]{ern2004theory}, to derive
\begin{equation*}
	\|\varphi_0\|_{H^1(\Omega)}^2
	\leq
	C(\|\Delta \varphi_0\|_{L^2(\Omega)}^2
 +
 \|\partial_{\boldsymbol{n}} \varphi_{0}\|_{L^{2}(\partial\Omega)}^2
 ),
\end{equation*}
which yields that
\begin{equation*}
\begin{aligned}
\mathfrak{B}(\boldsymbol{u},\boldsymbol{u}) + E \sqrt{\mathfrak{B}(\boldsymbol{u},\boldsymbol{u})}
    \geq\,&
    \frac{1}{2}\|\Delta\varphi_{1}\|_{L^2(\Omega)}^2 
        +
        \frac{\delta}{2} \big(\|\Delta \varphi_0\|_{L^2(\Omega)}^2 + \|\varphi_1\|_{L^2(\Omega)}^2 \big)\\
        &+
        \Big(\frac{\delta}{2C} - 2\delta^2 \Big)
        \|\varphi_0\|_{H^1(\Omega)}^2
        +
        \Big(\lambda-\frac{\delta}{2}\Big)\|\partial_{\boldsymbol{n}} \varphi_{0}\|_{L^2(\Omega)}^2.
        \end{aligned}
        \end{equation*}
Hence, by choosing \(0 < \delta \leq \min\{1/(8C), \lambda\}\), we verify the sup-linear coercivity of \(\mathfrak{B}(\boldsymbol{u},\boldsymbol{u})\).
\end{proof}

In our case, let \(X = H^2(\Omega)\times H^2(\Omega)\) and define the linear map \(T\) as
\begin{equation*}
    T = 
        (\mathcal{P},\, S),
\end{equation*}
where the matrix operator \(\mathcal{P}\) and the trace operator \(S\) are given by
\begin{equation}
    \mathcal{P}
	: =
	\begin{pmatrix}
		\Delta &                -I & \\
		0      &\Delta& 
	\end{pmatrix},
    \quad
    S = (I, \partial_{\boldsymbol{n}}).
\end{equation}
Moreover, we set
\begin{equation}\label{Def: Y space}
    Y= 
        L^2(\Omega)\oplus L^2(\partial\Omega), 
\end{equation}
so that \(T: X \to Y\). The linear map \(T\) then induces the bilinear form \(\mathfrak{B}\) defined in \eqref{bilinear form}, and the corresponding loss function \(\mathcal{L}\) is given by \eqref{eq_Lossfunctional}.  
Furthermore, define \(Z = \big(H^1(\Omega)\times H(\mathrm{div},\Omega)\big)\times \mathcal{H}(\Omega)\) with \(\mathcal{H}(\Omega):=\{u\in L^2(\Omega): \Delta u\in L^2(\Omega)\}\). Then Lemma \ref{lem_coercivity} implies that \(\mathfrak{B}\) is sup-linear coercive with respect to \(Z\).

Now we can apply the Céa's Lemma \ref{lem_cea} to the neural network approximation using the spaces and the bilinear form defined above. Suppose \(u^*\in \mathcal{B}^{6}(\Omega)\) is the solution to the biharmonic problem \eqref{eq_biharmonicequation}. By introducing the notation  
\[
\boldsymbol{u}^* = \big( u^*, \Delta u^*  \big),
\]  
it follows that \(\boldsymbol{u}^*\in H^2(\Omega)\times H^2(\Omega)\).  

Moreover, let  
\[
\boldsymbol{u}_{\theta} = (\varphi_0,\,\varphi_{1})
\]  
be any neural network in the product space \(\mathcal{F}_{\Theta} := \mathcal{F}_{m}(B)\times \mathcal{F}_{m}(B)\) with \(B = c\|u^*\|_{\mathcal{B}^6(\Omega)}\). 
For brevity, we define the approximation error between \(\boldsymbol{u}^*\) and \(\boldsymbol{u}_{\theta}\) and the generalization error between \(\mathcal{L}\) and \(\mathcal{L}_n\) as follows:  
\[
\left\{
\begin{aligned}
\mathcal{E}_{\rm app}
&= \inf_{\varphi_0,\varphi_1\in \mathcal{F}_{m}(B)} 
\sum_{k=0,1}\|\varphi_k-\Delta^{k}u^*\|_{H^2(\Omega)}^2, \\
\mathcal{E}_{\rm gen}
&= \sup_{\varphi_0,\varphi_1\in \mathcal{F}_{m}(B)} 
|\mathcal{L}(\boldsymbol{u}_\theta) - \mathcal{L}_n(\boldsymbol{u}_\theta)|,
\end{aligned}
\right.
\]  
where the empirical loss function \(\mathcal{L}_n\) is given by \eqref{eq_EmpiricalLoss}.  

Consequently, Lemma \ref{lem_coercivity} together with the Céa's Lemma \ref{lem_cea} and Remark \ref{rem_decompose} yield the following result.

\begin{lemma}\label{lem_generalerror}
Use the notation defined above and denote
    $$\boldsymbol{u}_{\theta^*}
=
(\varphi_0^*,\,  \varphi_1^*,)
=
\argmin_{\varphi_0, \varphi_1\in\mathcal{F}_m(B)}
\mathcal{L}_n(\varphi_0,\varphi_1),$$
    with $B = c\|u^*\|_{\mathcal{B}^6(\Omega)}$. Then it follows that
    \begin{equation*}
    \begin{aligned}
    \|\varphi_0^* - u^*\|_{H^1(\Omega)}^2 + &
			\|\varphi_1^* - \Delta u^*\|_{L^2(\Omega)}^2
    +
   \|\Delta(\varphi_0^* -  u^*) \|_{L^2(\Omega)}^2 
			 \\
            +&\, \|\Delta\varphi_1^* - \Delta^2 u^*\|_{L^2(\Omega)}^2
        \lesssim 
        \sqrt{\mathcal{E}_{\rm app}}
        +
         \sqrt{\mathcal{E}_{\rm gen}}.
        \end{aligned}
    \end{equation*}
\end{lemma}

From Lemma \ref{lem_generalerror} above, it follows that the error between the neural network approximation and the true solution can be bounded by the approximation error and the generalization error. Consequently, our task in what follows is to estimate these two errors.

\section{Error bounds estimation}\label{se4}
In this section, we estimate the error bounds of both the approximation error and the generalization error.

\subsection{Bounds of the approximation error}
The approximation properties of neural networks using the ReLU and ReQU activation functions have been thoroughly discussed in \cite{xurefined}. Here, we further extend these approximation results to the ReCU activation function, since the estimate of the generalization error requires the neural network to be three times differentiable with respect to the biharmonic equation.
\begin{lemma}\label{lem_sigma2approximation}
For any \( f \in \mathcal{B}^{4}(\Omega) \), there exists a constant \( c > 0 \) and a shallow neural network \( f_m \in \mathcal{F}_{m}\bigl(c\|f\|_{\mathcal{B}^{4}(\Omega)}\bigr) \) such that
\[
\|f - f_m\|_{H^{3}(\Omega)} \leq c\|f\|_{\mathcal{B}^{4}(\Omega)} \, m^{-\bigl(\frac{1}{2} + \frac{1}{3d}\bigr)},
\]
where the function class \( \mathcal{F}_m \) is defined in \eqref{eq_networkfunctionclass}.
\end{lemma}
For brevity, the proof of this lemma is provided in Appendix \ref{se_appendix}.

\subsection{Estimation of the Rademacher Complexity}
We will illustrate that the generalization error for a model trained on a finite dataset $\{X_i\}_{i=1}^{n}$ can be estimated using the Rademacher complexity.  
\begin{definition}
	Let $X=\{X_i\}_{i=1}^{n}$ be a set of random variables in $\Omega$ that is independently distributed, and $\sigma = \{\sigma_i\}_{i=1}^{n}$ be independent Rademacher random variables that take values 
	$+1$ or $-1$ with equal probability. Then the \textbf{empirical Rademacher Complexity} of the function class $\mathcal{S}$ is a random variable given by
	\begin{equation*}
		\hat{R}_{n}(\mathcal{S}):=
		\mathbb{E}_\sigma\Big[\sup_{f\in \mathcal{S}}\Big|\frac{1}{n}\sum_{i=1}^n\sigma_i f(X_i)\Big|\Big].
	\end{equation*}
	Taking its expectation in terms of $X$ yields the \textbf{Rademacher Complexity} of the function class $\mathcal{S}$
	\begin{equation*}
		R_{n}(\mathcal{S})
		=
		\mathbb{E}_X\mathbb{E}_\sigma\Big[\sup_{f\in \mathcal{S}}\Big|\frac{1}{n}\sum_{i=1}^{n}\sigma_i f(X_i)\Big|\Big].
	\end{equation*}
\end{definition}
Rademacher complexity is a fundamental tool in bounding the generalization error. In this work, we will employ several key lemmas and theorems associated with it in \cite{li2024priori,ledoux2013probability}. 

The following lemma is obtained by combining \cite[Lem.~26.11]{shalev2014understanding} and \cite[Lem.~4.4]{li2024priori}.
\begin{lemma} \label{Lem: linear transformation function class}
   Let $\mathcal{G}$ be the linear transformation function class defined by 
   \begin{equation}\label{linear transformation function class}
       \mathcal{G}
       :=
       \{\boldsymbol{\omega}\cdot\boldsymbol{x} + t
       :
       |\boldsymbol{\omega}|_1=1,\,|t|\leq 1
       \}.
   \end{equation}
   Then there holds
   \begin{equation*}
       R_n(\mathcal{G})
       \leq
        \frac{\sqrt{ 2\log(2d) }+1}{\sqrt{n}}.
   \end{equation*}
\end{lemma}

With the associated property of Rademacher complexity, we are ready to estimate the complexity of the neural network function class \( \mathcal{F}_m \) defined in \eqref{eq_networkfunctionclass} below. 

\begin{lemma}
    The Rademacher complexity of $\mathcal{F}_m(B)$ is bounded by
    \begin{equation*}
  R_{n}(\mathcal{F}_{m}(B)) 
  \leq
  \frac{24B(\sqrt{ 2\log(2d) }+1)}{\sqrt{n}}.
    \end{equation*}
\end{lemma}
\begin{proof}
   By utilizing the property from \cite[Lem.~4.4]{li2024priori}, the Rademacher complexity of shallow neural networks $\mathcal{F}_m(B)$ defined in \eqref{eq_networkfunctionclass} can be bounded by
   \begin{equation*}
       R_n(\mathcal{F}_m(B))
       \leq
       \frac{1}{m} \sum_{i=1}^m |\gamma_i|R_n(\sigma_3(\boldsymbol{\omega}_1\cdot \boldsymbol{x} + t_i))
       :=
       \frac{1}{m} \sum_{i=1}^m |\gamma_i|R_n(\sigma_3(\mathcal{G})),    \end{equation*}
    where $\mathcal{G}$ is the linear transformation function class defined in \eqref{linear transformation function class}. Since $\sigma_3$ is $12$-Lipschitz and $\sigma_3(0)=0$. Then applying Lemma \ref{Lem: linear transformation function class} and contraction property, see \cite[Lem.~4.4]{li2024priori} , we have
    \begin{equation*}
        R_n(\mathcal{F}_m(B)) 
        \leq
        \frac{24}{m} \sum_{i=1}^m |\gamma_i|R_n(\mathcal{G})
        \leq
        \frac{24B(\sqrt{ 2\log(2d) }+1)}{\sqrt{n}}.
    \end{equation*}
\end{proof}
   
In the following, we explore the relationship between the generalization error and the Rademacher complexity. Note that estimating the complexity of a shallow neural network depends on the choice of the activation function. Since $\sigma = \mathrm{ReCU}$ is three times continuously differentiable almost everywhere, we can make the following assumptions:
\begin{equation*}
    \sup |\sigma^{(k)}|\leq \ell_k,
    \quad
    k=0,\,1\, ,2\, ,\, 3.
\end{equation*} 
To bound the generalization error, we estimate the Rademacher complexity for the following function classes. Define 
$$K=\max\{\|u^*\|_{\mathcal{B}^6(\Omega)}^2,\, \|f\|_{L^\infty(\Omega)}, \,\|g\|_{L^\infty(\partial\Omega)},\,\|h\|_{L^\infty(\partial\Omega)}\}.$$
For $\boldsymbol{u} = (\varphi_0, \varphi_1)$, let $l(\boldsymbol{u}) = |\Delta \varphi_1 - f|^2 + |\varphi_1 - \Delta \varphi_0|^2$, and define the interior loss function class:
\begin{equation}
	\begin{aligned}
		\mathcal{S}
		=
		\big\{|\Delta \varphi_1 - f|^2 + |\varphi_1 - \Delta \varphi_0|^2 \,
		\big|\, \varphi_0, \varphi_1 \in \mathcal{F}_m(B)\big\}.
	\end{aligned}
\end{equation}
Similarly, let $l(\boldsymbol{u}) = |\varphi_0 - g|^2 + |\partial_{\boldsymbol{n}} \varphi_1 - h|^2$, and define the boundary loss function class:
\begin{equation}
	\begin{aligned}
		\mathcal{S}_{\mathrm{b}}
		=
		\big\{
		|\varphi_0 - g|^2 + |\partial_{\boldsymbol{n}} \varphi_1 - h|^2 \,\big|\, \varphi_0, \varphi_1\in \mathcal{F}_m(B)\big\}.
	\end{aligned}
\end{equation}

We then establish the following estimate for the Rademacher complexity. 
\begin{lemma}\label{lem_RademacherComplexity2} Let \( n \) and \( \widehat{n} \) denote the numbers of sample points in \( \Omega \) and on \( \partial\Omega \), respectively. Assume \( \widehat{n} = O(n) \).  
The function classes \( \mathcal{S} \) and \( \mathcal{S}_{\mathrm{b}} \) satisfy the following complexity bound:
	\begin{equation*}
		R_{n}(\mathcal{S})
		+
		R_{\widehat{n}}(\mathcal{S}_{\mathrm{b}})
		\leq
		\frac{C K^2}{\sqrt{n}},
	\end{equation*}
where \( C > 0 \) depends polynomially on the dimension \( d \).
\end{lemma}

\begin{proof}
According to the definition of the Rademacher complexity, it follows directly that the upper bound of $\mathcal{S}$ can be divided into the following parts:
\begin{equation*}
\left\{
    \begin{aligned}
    \mathcal{S}_1
       =&
       \{ |\Delta \varphi_1 - f|^2 \,\big|\, \varphi_1 \in \mathcal{F}_{m}(B)\},\\
       \mathcal{S}_2
       =&
       \{ |\Delta \varphi_0 - \varphi_1|^2 \,\big|\, \varphi_0,\,\varphi_1 \in \mathcal{F}_{m}(B)\}.
    \end{aligned}
    \right.
\end{equation*}
Here we set $B=c\|u^*\|_{\mathcal{B}^6(\Omega)}$. 
According to property from \cite[Lem.~4.4]{li2024priori}, we obtain
\begin{equation*}
    R_n(\mathcal{S})
    \leq
    R_n(\mathcal{S}_1) + R_n(\mathcal{S}_2).
\end{equation*}
Next, we estimate each term in the above inequality.
The first term appears as a gradient penalty term in the loss functions of both problems. Since $\Delta \varphi_1 = \sum_{j=1}^d\partial_{x_j x_j} \varphi_1 $, viewing each dimension separately yields, for $j=1,\,2,\,\cdots,\, d$,
\begin{equation*}
\begin{aligned}
    \partial_{x_j}\varphi_1
    =
    \sum_{i=1}^m \gamma_i \omega_{i,j}\sigma_3^{(1)}(\boldsymbol{\omega}_i\cdot \boldsymbol{x} + t_i), \mbox{ and }
    \partial_{x_j x_j}\varphi_1
    =
    \sum_{i=1}^m \gamma_i |\omega_{i,j}|^2 \sigma_3^{(2)}(\boldsymbol{\omega}_i\cdot \boldsymbol{x} + t_i)
    \end{aligned}
\end{equation*}
where $ \omega_{i,j}$ denotes the $j$-th component of $\boldsymbol{\omega}_i$. Then it follows that 
\begin{equation*}
    \begin{aligned}
    \Delta \varphi_1 
       =
    \sum_{j=1}^d\sum_{i=1}^m \gamma_i |\omega_{i,j}|^2 \sigma_3^{(2)}(\boldsymbol{\omega}_i\cdot \boldsymbol{x} + t_i)
    \leq 
    \sum_{i=1}^m \gamma_i |\omega_{i}|^2_1 \sigma_3^{(2)}(\boldsymbol{\omega}_i\cdot \boldsymbol{x} + t_i),
    \end{aligned}
\end{equation*}
which implies that 
\begin{equation*}
    \|\Delta \varphi_1 - f \|_{L^\infty(\Omega)}
    \leq\,
    \|\Delta \varphi_1  \|_{L^\infty(\Omega)} + 
    \| f \|_{L^\infty(\Omega)}
    \leq\,
    C (\ell_2 + 1)K.
\end{equation*}
Hence, we deduce that $R_n(\mathcal{S}_1)$ can be bounded by
\begin{equation*}
    \begin{aligned}
        4\sup_{\varphi_0, \varphi_1\in \mathcal{F}_{m}(B)
     }&
        \|\Delta \varphi_1 - f\|_{L^\infty(\Omega)}
        \big( R_n( f )
        +
        \sum_{j=1}^d\sum_{i=1}^m|\gamma_i| |\omega_{i,j}|^2R_n (\sigma_3^{(2)}(\mathcal{G}))
        \big)\\
        \leq& 
        \frac{cK(B(\sqrt{\log (2d)}+1) + \|f\|_{L^\infty})}{\sqrt{n}}
         \leq
        \frac{C K^2(\sqrt{\log(2d)}+1)}{\sqrt{n}}, 
    \end{aligned}
\end{equation*}
where we have used the inequality
\begin{equation*}
\begin{aligned}
\sum_{j=1}^d\sum_{i=1}^m|\gamma_i| |\omega_{i,j}|^2 R_n (\sigma_3^{(1)}(\mathcal{G}))
\leq&
\sum_{i=1}^m|\gamma_i| |\boldsymbol{\omega}_{i}|_1^2 R_n (\sigma_3^{(1)}(\mathcal{G}))\\
\leq&
2 B\ell_3 R_n (\mathcal{G})
\leq
\frac{CB(\sqrt{\log (2d)}+1)}{\sqrt{n}}.
\end{aligned}
\end{equation*}
In addition, for the function class $\mathcal{S}_2$, we know that $\|\varphi_1\|_{L^\infty} \leq C B$, which yields, in a similar way,
\begin{equation}
    \begin{aligned}
        R_{n}(\mathcal{S}_2)
        \leq\,&
        4\sup_{\varphi_0, \varphi_1\in \mathcal{F}_m(B)}
        \| \Delta\varphi_0 - \varphi_1 \|_{L^\infty(\Omega)}
        (2 B\ell_3 R_n (\mathcal{G})  + R_n(\mathcal{F}_m))\\
        \leq\,&
        \frac{CB^2\sqrt{\log(2d)}}{\sqrt{n}}.
    \end{aligned}
\end{equation}
The last step is to estimate the function classes concerning the boundary conditions. We denote
\begin{equation*}
    \left\{
    \begin{aligned}
        S_{\mathrm{b},1}
        =&
        \{|\varphi_0 - g|^2 : \varphi_0 \in \mathcal{F}_m(B) \},\\
        S_{\mathrm{b},2}
        =&
        \{|\partial_{\boldsymbol{n}}\varphi_0 - h|^2 : \varphi_0 \in \mathcal{F}_m(B) \}.
    \end{aligned}
    \right.
\end{equation*}
Similarly, we obtain
\begin{equation*}
\begin{aligned}
    R_{\widehat{n}}(S_{\mathrm{b},1})
    \leq\,&
    4\sup_{\varphi_0 \in \mathcal{F}_m(B)}
        \| \varphi_0 - g\|_{L^\infty(\partial\Omega)}
        ( R_{\widehat{n}} (\mathcal{G})  + R_{\widehat{n}}(g))\\
        \leq\,&
        \frac{CK^2(\sqrt{\log(2d)} + 1)}{\sqrt{n}}.
        \end{aligned}
\end{equation*}
On the other hand, since 
$$\partial_{\boldsymbol{n}}\varphi_0 = \nabla\varphi_0 \cdot \boldsymbol{n} = \sum_{j=1}^d  n_j\, \partial_{x_j}\varphi_0
    =
    \sum_{j=1}^d  n_j\, \sum_{i=1}^m \gamma_i \omega_{i,j}\sigma_3^{(1)}(\boldsymbol{\omega}_i\cdot \boldsymbol{x} + t_i),$$
it follows that $\|\partial_{\boldsymbol{n}}\varphi_0\|_{L^\infty(\partial\Omega)} \leq 
C B$. Hence we have
\begin{equation*}
\begin{aligned}
    R_{\widehat{n}}(\mathcal{S}_{\mathrm{b},2})
    \leq\,&
    4\sup_{\varphi_0\in \mathcal{F}_m(B)} \|\partial_{\boldsymbol{n}}\varphi_0- h\|_{L^\infty(\partial\Omega)}
    (\sum_{j=1}^d\sum_{i=1}^m|\gamma_i| |\omega_{i,j}| R_n (\sigma_3^{(1)}(\mathcal{G})) + R_{\widehat{n}}(h)\\
    \leq\,&
    \frac{C K^2(\sqrt{\log(2d)}+1)}{\sqrt{\widehat{n}}}.
    \end{aligned}
\end{equation*}
In general, the number of sample points on the boundary is not the same as that in the interior domain. Setting $\widehat{n}=O(n)$ then leads to the upper bound on the Rademacher complexity.
\end{proof}

\subsection{Proof of Theorem \ref{thm2}}
Now we are ready to apply the preceding results to prove Theorem \ref{thm2}.

\begin{proof}[Proof of Theorem \ref{thm2}]
From Lemma \ref{lem_generalerror}, we conclude that the total errors $\|\varphi_0^n - u^*\|_{H^1(\Omega)}^2$, $\|\varphi_1^* - \Delta u^*\|_{L^2(\Omega)}^2$, and $\|\Delta(\varphi_0^* - \Delta^k u^*) \|_{L^2(\Omega)}^2$ for $k=0,1$ can be bounded by the generalization error
\begin{equation*}
    \mathcal{E}_{\rm gen}
       =
       \sup_{\varphi_0,\varphi_1\in \mathcal{F}_{m}(B)}
       |\mathcal{L}(\boldsymbol{u_\theta})
       -
       \mathcal{L}_n(\boldsymbol{u}_\theta)|,
\end{equation*}
        and the approximation error
\begin{equation*}
            \mathcal{E}_{\rm app}
        =
        \inf_{\varphi_0,\varphi_1\in \mathcal{F}_{m}(B)}
    \sum_{k=0,1} \|\varphi_k-\Delta^{k}u^*\|_{H^2(\Omega)}^2.
        \end{equation*}
By Lemma \ref{lem_sigma2approximation}, since $u^*\in \mathcal{B}^{6}(\Omega)$, we derive the following estimate for the approximation error:
\begin{equation*}
        \mathcal{E}_{\rm app}
        \leq 
        C\|u^*\|_{\mathcal{B}^{4}(\Omega)}^2m^{-(1 + \frac{2}{3d})}.
\end{equation*}
Here we set $B=c\|u^*\|_{\mathcal{B}^{6}(\Omega)}^2$ for a constant $c>0$. Furthermore, for the generalization error, applying Lemma \ref{lem_RademacherComplexity2} together with the relation between generalization error and Rademacher complexity, see \cite[Lem.~4.8]{li2024priori}, we deduce that the following estimate holds for $\varphi_k \in \mathcal{F}(B)$:
\begin{equation*}
    \mathbb{E}|\mathcal{L}(\boldsymbol{u}_\theta)
       -
       \mathcal{L}_n(\boldsymbol{u}_\theta)|
       \leq\,
        C\big( R_{n}(\mathcal{S})
        +
        R_{\widehat{n}}(\mathcal{S}_b) \big)
        \leq\,
        \frac{C K^2}{\sqrt{n}},
\end{equation*}
where the expectation is taken with respect to the random sampling of training data in $\Omega$ and on $\partial\Omega$, and we take $\widehat{n}=O(n)$. The constant $C$ appearing in the inequality depends at most polynomially on $d$. 

Finally, applying Lemma \ref{lem_generalerror}, we have thus completed the proof of the theorem.
\end{proof}

\section{Numerical Experiments}\label{se5}
In this section, we present numerical experiments to verify our theoretical results and compare the practical performance of different neural network methods for high-dimensional biharmonic equations. We emphasize that the aforementioned theoretical analysis is based on the empirical loss function (\ref{eq_EmpiricalLoss}), which will be denoted as the $\mathrm{MIM_a}$ method in this section.

We consider the biharmonic equation with clamped boundary conditions in the domain $\Omega = (-1,1)^d$:
\begin{equation*}
     \left\{
 \begin{aligned}
     	&\Delta^2 u
         =
         \frac{\pi^4}{16} \sum_{k=1}^d \sin \left(\frac{\pi  x_k}{2} \right), \quad \text{in } \Omega, \\
         &u
         =
         \sum_{k=1}^d \sin \left(\frac{\pi x_k}{2}\right),\quad \partial_{\boldsymbol{n}}u=0, \quad \text{on } \partial\Omega.  
 \end{aligned}
     \right.
 \end{equation*}
The exact solution is $u^*(\boldsymbol{x}) = \sum_{k=1}^d \sin (\frac{\pi x_k}{2})$. Throughout this section, we use the PINN method from \cite{raissi2019physics} as a baseline algorithm. Moreover, we introduce another mixed residual method, denoted $\mathrm{MIM_b}$, which uses a neural network $\boldsymbol{u} = (\phi_0,\, \boldsymbol{\psi}_0,\, \phi_1,\, \boldsymbol{\psi}_1)$ to approximate the true solution $\boldsymbol{u}^* = (u^*,\, \nabla u^*,\, \Delta u^*,\, \nabla\Delta u^*)$. The expected loss function of $\mathrm{MIM_b}$ is defined as:
\begin{equation*}
\begin{aligned}
\mathcal{L}_{\mathrm{MIM_b}}
 =&
 \| \bdiv\, \boldsymbol{\psi}_{1} - f \|_{L^2(\Omega)}^2 
 + 
 \| \boldsymbol{\psi}_1 - \nabla\phi_1 \|_{L^2(\Omega)}^2  
 +
 \| \phi_1 - \bdiv\,\boldsymbol{\psi} _0 \|_{L^2(\Omega)}^2  \\
 &+
 \| \boldsymbol{\psi}_0 - \nabla\phi _0 \|^2_{L^2(\Omega)}
 + \lambda \| {{\phi _0} - g} \|^2_{{L^2}(\partial \Omega )} + \lambda {\| {{\boldsymbol{\psi} _0} \cdot \boldsymbol{n}} \|^2_{{L^2}(\partial \Omega )}},
 \end{aligned}
 \end{equation*}
where $\lambda$ is the weight of the boundary loss, and $g = \sum_{k=1}^d \sin \left(\frac{\pi x_k}{2}\right)$. In practice, the empirical loss $\mathcal{L}_{\mathrm{MIM_b},n}$ is minimized to approximate the solution of the PDE.

\begin{table}[htbp]
\centering
\caption{Relative $L^2$ errors ($\times 10^{-4}$) of $u$ and its derivatives for different methods in dimension $d=2$.}
\label{tab:rel_error_d2}
\setlength{\tabcolsep}{6pt}
\renewcommand{\arraystretch}{1.2}
\begin{tabular}{c c c c c c c c }
\hline
Activation & Method & $u$ & $\nabla u$ & $\Delta u$ & $\nabla(\Delta u)$ & $\Delta^2 u$  & Time(s)\\
\hline
\multirow{3}{*}{ReCU}
  &$\mathrm{MIM_a}$ & 0.935 & 4.871 & 5.160 & 21.036 & 13.702 & 882.95 \\
 & $\mathrm{MIM_b}$ & 2.607 & 8.148 & 9.435 & 31.962 & 1.245 & 1941.56\\
 & PINN & - & -  & - & - & - & - \\
\hline
\multirow{3}{*}{$\rm ReLU^5$}
 & $\mathrm{MIM_a}$ & 0.561 & 2.666 & 5.379 & 23.326 & 1.243 & 999.31 \\
 & $\mathrm{MIM_b}$ & 1.190 & 4.358 & 13.262 & 56.799 & 0.603 & 3290.34\\
 & PINN & 0.235 & 0.645 & 2.026 & 9.148 & 19.555 & 3451.98 \\
\hline
\end{tabular}
\end{table}


\begin{table}[htbp]
\centering
\caption{Relative $L^2$ errors ($\times 10^{-4}$) of $u$ and its derivatives for different methods in dimension $d=5$.}
\label{tab:rel_error_d5}
\setlength{\tabcolsep}{6pt}
\renewcommand{\arraystretch}{1.2}
\begin{tabular}{c c c c c c c c }
\hline
Activation & Method & $u$ & $\nabla u$ & $\Delta u$ & $\nabla(\Delta u)$ & $\Delta^2 u$  & Time(s)\\
\hline
\multirow{3}{*}{ReCU}
  & $\mathrm{MIM_a}$ & 0.424 & 2.197 & 7.792 & 28.502 & 46.176 & 4298.10\\
 & $\mathrm{MIM_b}$ & 2.538 & 10.967  & 24.144 & 74.537 & 3.185 & 5622.52 \\
 & PINN & - & -  & - & - & - & - \\
\hline
\multirow{3}{*}{$\rm ReLU^5$}
 & $\mathrm{MIM_a}$ & 0.530 & 2.555 & 6.680 & 22.682 & 1.593 & 2396.09 \\
 & $\mathrm{MIM_b}$ & 0.689 & 27.684 & 57.737 & 206.121 & 0.350 & 4362.16\\
 & PINN & 0.462 & 1.161 & 3.339 & 12.362 & 51.044 & 4599.15 \\
\hline
\end{tabular}
\end{table}

\begin{table}[htbp]
\centering
\caption{Relative $L^2$ errors ($\times 10^{-4}$) of $u$ and its derivatives for different methods in dimension $d=6$.}
\label{tab:rel_error_d6}
\setlength{\tabcolsep}{6pt}
\renewcommand{\arraystretch}{1.2}
\begin{tabular}{c c c c c c c c }
\hline
Activation & Method & $u$ & $\nabla u$ & $\Delta u$ & $\nabla(\Delta u)$ & $\Delta^2 u$  & Time(s)\\
\hline
\multirow{3}{*}{ReCU}
  & $\mathrm{MIM_a}$ & 0.515 & 3.914 & 7.939 & 22.771 & 36.145 & 4806.66\\
 & $\mathrm{MIM_b}$ & 2.548  & 12.395 & 26.421 & 81.319 & 5.627 & 7859.84\\
 & PINN & - & -  & - & - & - & - \\
\hline
\multirow{3}{*}{$\rm ReLU^5$}
 & $\mathrm{MIM_a}$ & 0.522 & 2.440 & 7.689 & 24.735 & 1.986 & 3073.42 \\
 & $\mathrm{MIM_b}$ & 2.616 & 18.194 & 29.037 & 91.642 & 3.350& 4039.33\\
  & PINN & 0.336 & 0.981 & 2.466 & 9.694 & 61.077 & 6536.48 \\
\hline
\end{tabular}
\end{table}


\begin{table}[htbp]
\centering
\caption{Relative $L^2$ errors ($\times 10^{-4}$) of $u$ and its derivatives for different methods in dimension $d=8$.}
\label{tab:rel_error_d8}
\setlength{\tabcolsep}{6pt}
\renewcommand{\arraystretch}{1.2}
\begin{tabular}{c c c c c c c c }
\hline
Activation & Method & $u$ & $\nabla u$ & $\Delta u$ & $\nabla(\Delta u)$ & $\Delta^2 u$  & Time(s)\\
\hline
\multirow{3}{*}{ReCU}
 & $\mathrm{MIM_a}$ & 0.887 & 6.595 & 22.365 & 62.518 & 75.919 & 6790.71\\
 & $\mathrm{MIM_b}$ & 2.764  & 22.346 &48.818 & 147.842 & 5.319 & 7216.25\\
 & PINN & - & -  & - & - & - & - \\
\hline
\multirow{3}{*}{$\rm ReLU^5$}
   & $\mathrm{MIM_a}$ & 0.658 & 3.919 & 18.651 & 54.091 & 11.605 & 6146.98 \\
 & $\mathrm{MIM_b}$ & 2.448 & 9.438  & 19.157 & 55.123 & 10.870& 7450.41\\
  & PINN & 0.443 & 1.281 & 3.207 & 12.811 & 73.363 & 7828.70\\
\hline
\end{tabular}
\end{table}


For the numerical experiments, we employ a shallow feedforward neural network. ReCU and $\mathrm{ReLU^5}$ are selected as the activation functions for $\mathrm{MIM_a}$, $\mathrm{MIM_b}$ and PINN, respectively. Tables \ref{tab:rel_error_d2}–\ref{tab:rel_error_d8} summarize the relative $L^2$ errors of the neural network approximation and their derivatives across dimensions $d \in \{2,\, 5,\, 6,\, 8\}$. Additionally, the final column reports the computational time required for $50,000$ training epochs. To ensure a fair comparison across methods, key hyperparameters are kept constant, including the network width ($m=200$) and the spatial discretization ($8000$ interior points and $6000$ boundary points). Note that these timing results are provided primarily for reference, as execution time is hardware-dependent and subject to inherent randomness.

As observed from Tables \ref{tab:rel_error_d2}–\ref{tab:rel_error_d8}, while PINN requires activation functions with higher regularity, it achieves accuracy merely comparable to that of $\mathrm{MIM_a}$ for the neural network approximation of both the solution and its derivatives, at the expense of higher computational cost. Notably, when utilizing the identical $\mathrm{ReLU^5}$ activation function, the accuracy of $\mathrm{MIM_a}$ for the fourth-order derivative improves by an order of magnitude, as the error is reduced from $10^{-3}$ to $10^{-4}$ in most cases. Furthermore, the comparison between $\mathrm{MIM_a}$ and $\mathrm{MIM_b}$ reveals that increasing the degree of system splitting does not strictly translate into improved overall accuracy. Although $\mathrm{MIM_b}$ effectively mitigates the error associated with the highest-order operator, it tends to amplify errors in lower-order quantities in high-dimensional settings. Consequently, balancing the trade-offs among these split terms presents a new set of challenges. This demonstrates that the specific splitting strategy employed in $\mathrm{MIM_a}$ strikes an optimal balance: it adequately reduces the required regularity of the solution without overly fragmenting the system, thereby preserving the correlation between lower- and higher-order terms. Notably, as the dimensionality increases from $d=2$ to $d=8$, $\mathrm{MIM_a}$ exhibits robustness and scalability, maintaining a consistent error magnitude without suffering from severe exponential growth in computational overhead.
\begin{figure}[htbp]
    \centering
    \begin{subfigure}[b]{0.45\textwidth}
        \centering
        \includegraphics[width=\linewidth]{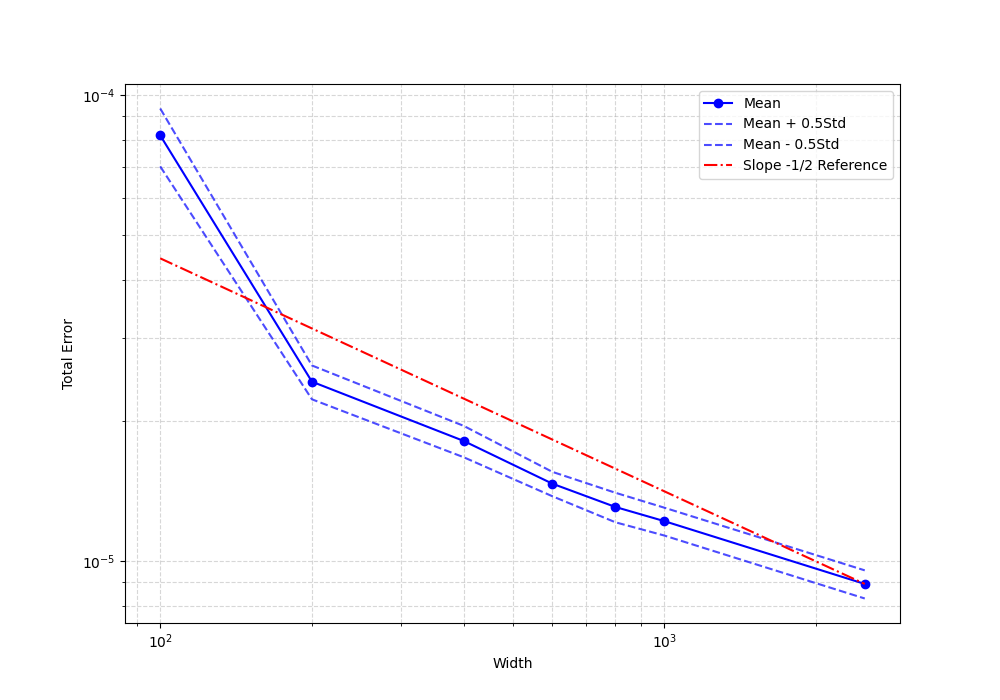} 
        \caption{}
    \end{subfigure}
    \hfill
    \begin{subfigure}[b]{0.45\textwidth}
        \centering
        \includegraphics[width=\linewidth]{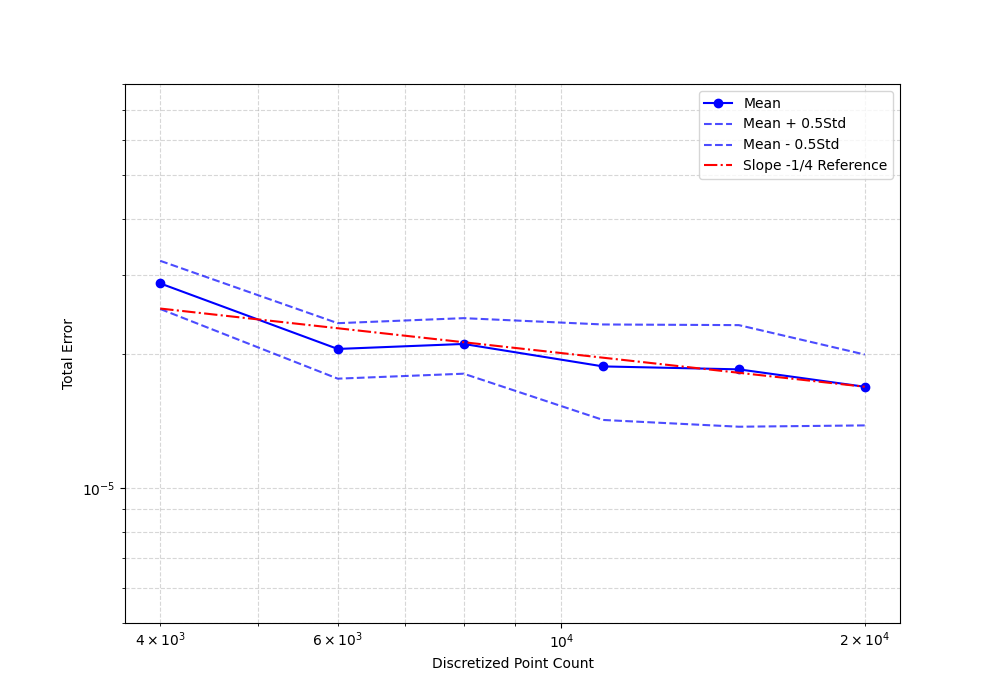} 
        \caption{}
    \end{subfigure}
    \caption{Log-log plots of the total error under the ReCU activation function. Panel (A) shows the decay with increasing network width, while panel (B) illustrates the convergence with respect to the number of sampling points.}
     \label{fig:log-log}
\end{figure}

To further validate the theoretical approximation rate in Theorem~\ref{thm2}, we present the total error with respect to the number of neurons or sampling points on a log-log scale in Figure~\ref{fig:log-log}. It can be observed that the approximation rates exhibit a nearly linear behavior when either the number of sampling points or the network width is fixed. This observation aligns uniformly with the theoretical error bound established in Theorem~\ref{thm2}. More precisely, for the network width study, experiments are conducted on a 2D domain with 6,000 randomly sampled interior points and 1,500 boundary points, testing widths of 100, 200, 400, 600, 800, 1,000, and 2,500. For the sampling point study, we fix the network width at 200 and the number of boundary points at 6,000, while varying the number of interior points as follows: 4,000, 6,000, 8,000, 11,000, 15,000, and 20,000. Each configuration is first trained using the Adam optimizer, for 10,000 iterations in the width study and 15,000 iterations in the sampling point study,  followed by LBFGS optimization. To mitigate the effect of optimizer randomness, we perform 10 independent runs with different random seeds.

\begin{figure}[htbp]
\centering

\begin{subfigure}[t]{0.5\textwidth}
  \centering
  \includegraphics[width=\linewidth]{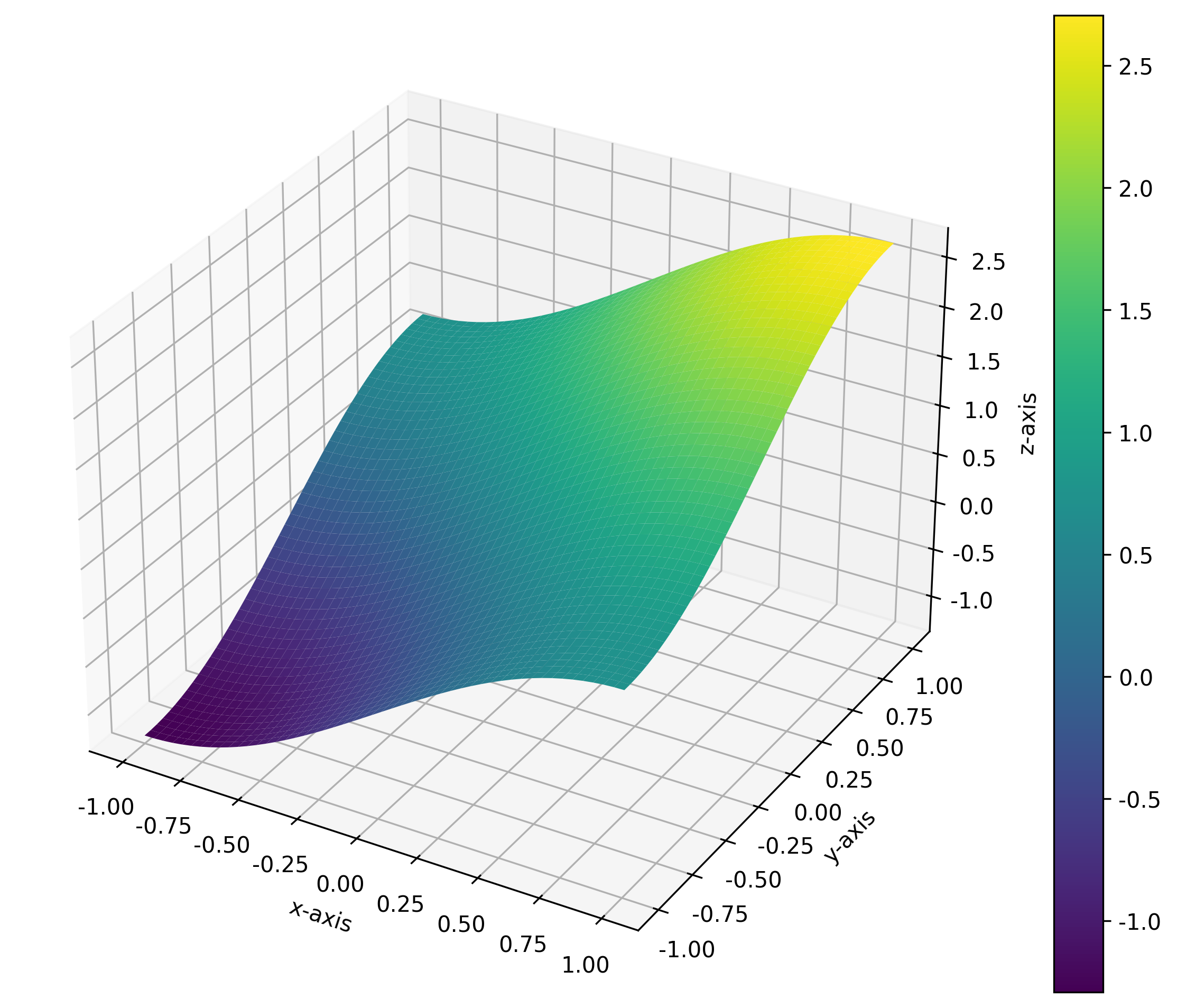}
  \caption{Exact}
\end{subfigure}

\vspace{0.5em}

\begin{subfigure}[t]{0.32\textwidth}
  \centering
  \includegraphics[width=\linewidth]{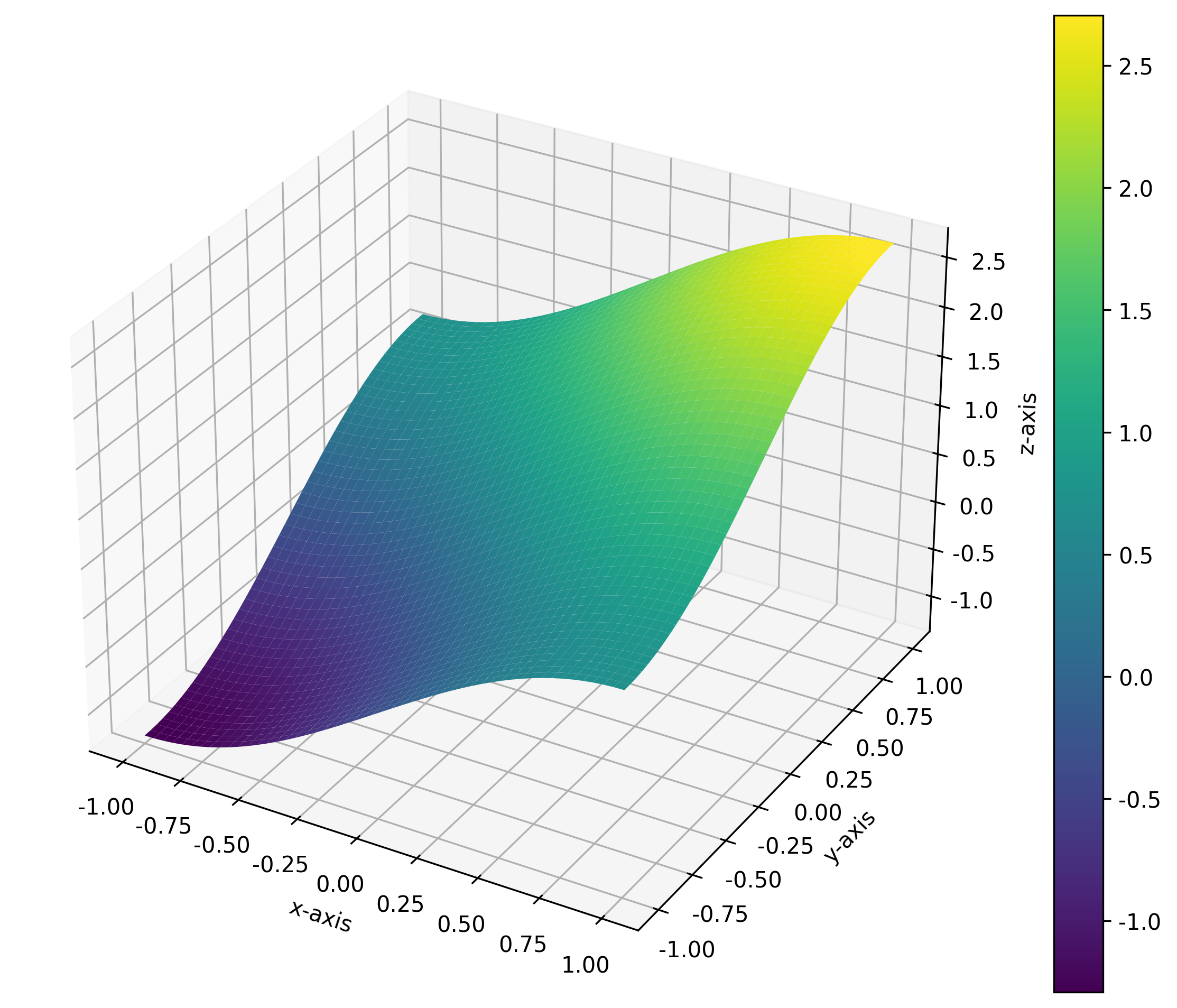}
  \caption{$\rm MIM_a$}
\end{subfigure}\hfill
\begin{subfigure}[t]{0.32\textwidth}
  \centering
  \includegraphics[width=\linewidth]{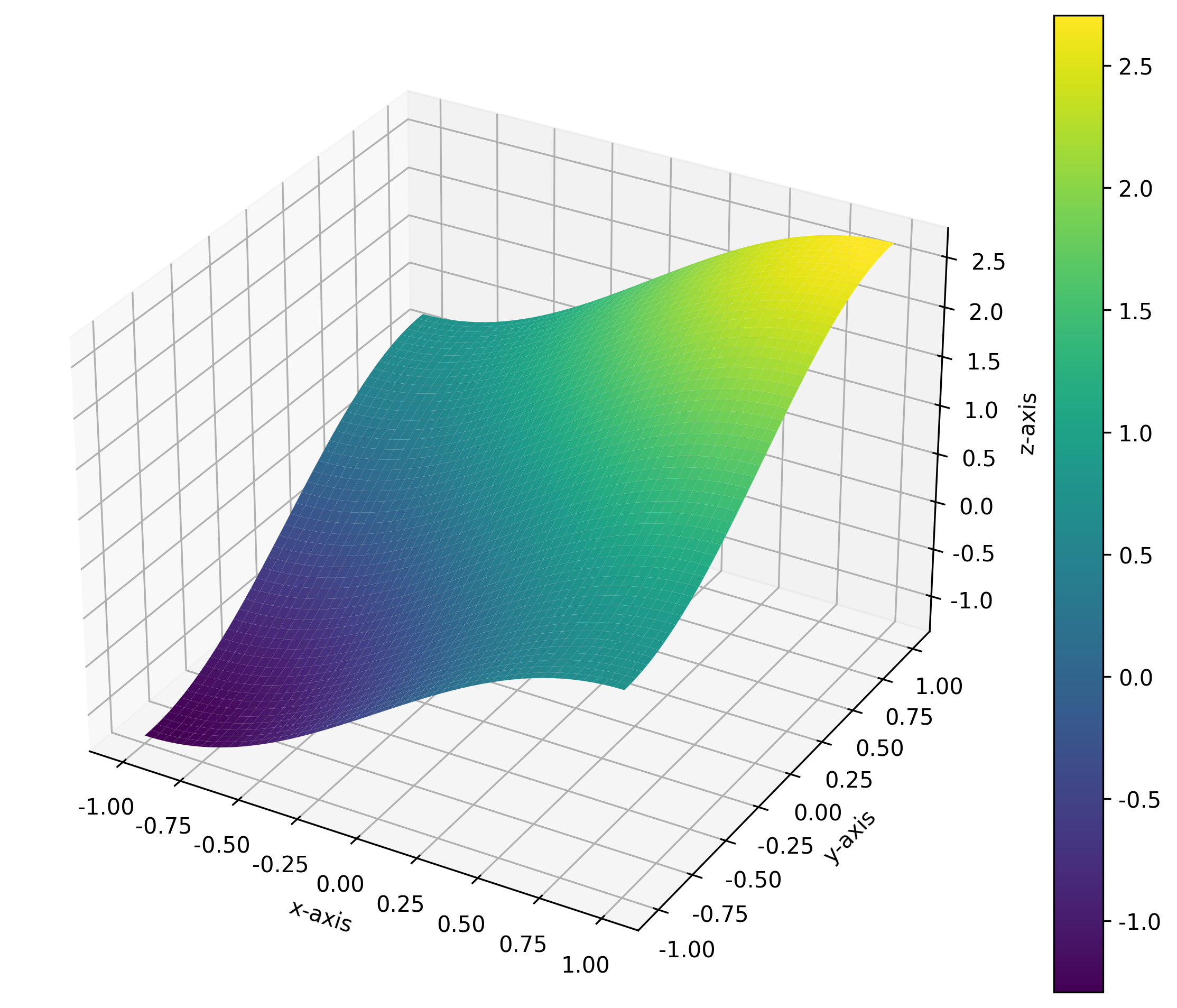}
  \caption{$\rm MIM_b$}
\end{subfigure}\hfill
\begin{subfigure}[t]{0.32\textwidth}
  \centering
  \includegraphics[width=\linewidth]{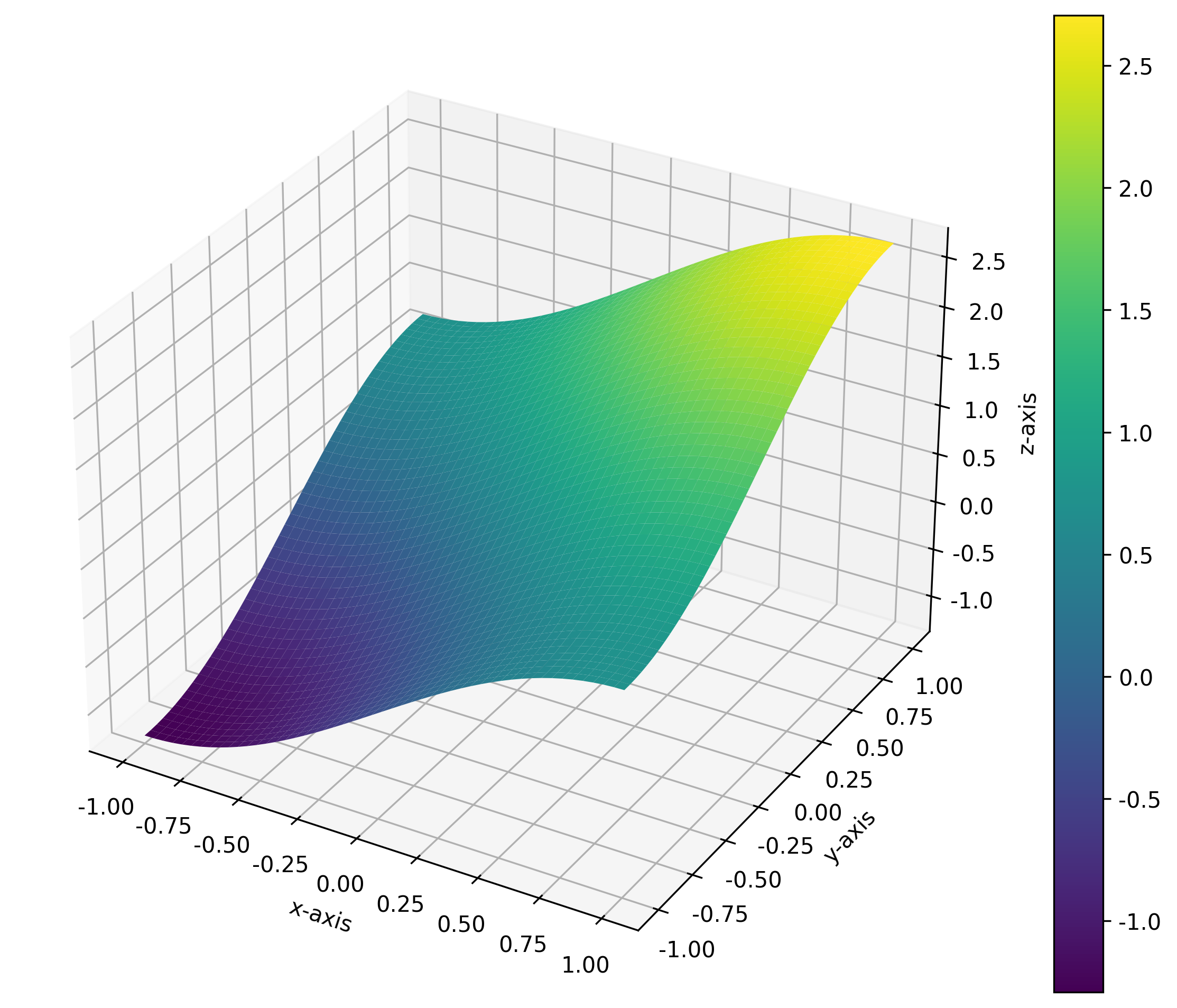}
  \caption{PINN}
\end{subfigure}

\caption{Slice plots of the exact solution and numerical solutions obtained by three methods on the plane $x_3 = 0.5$.}
\label{fig:slice-solutions}
\end{figure}

\begin{figure}[htbp]
\centering

\begin{subfigure}[t]{0.32\textwidth}
  \centering
  \includegraphics[width=\linewidth]{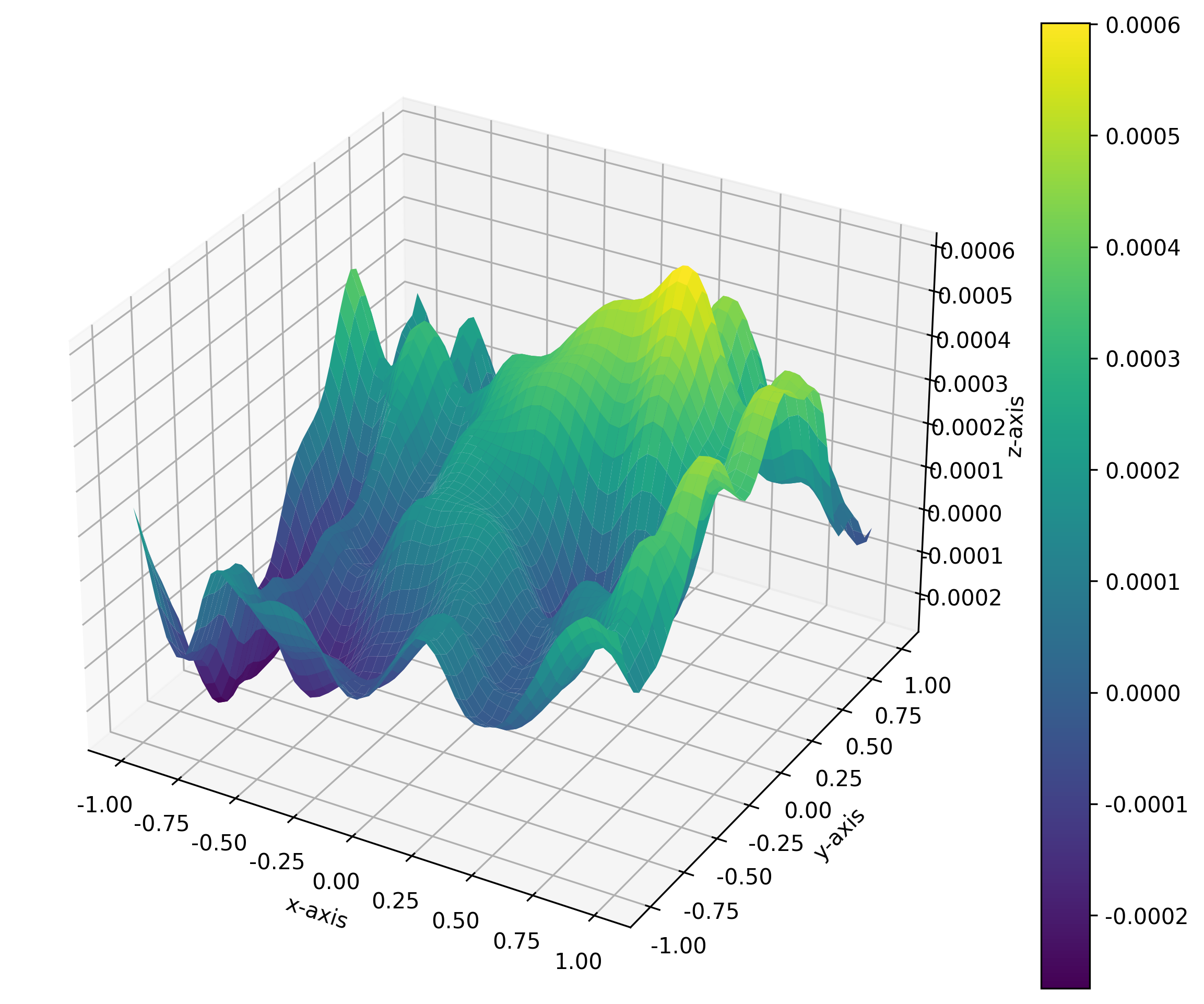}
  \caption{$\rm MIM_a$}
\end{subfigure}\hfill
\begin{subfigure}[t]{0.32\textwidth}
  \centering
  \includegraphics[width=\linewidth]{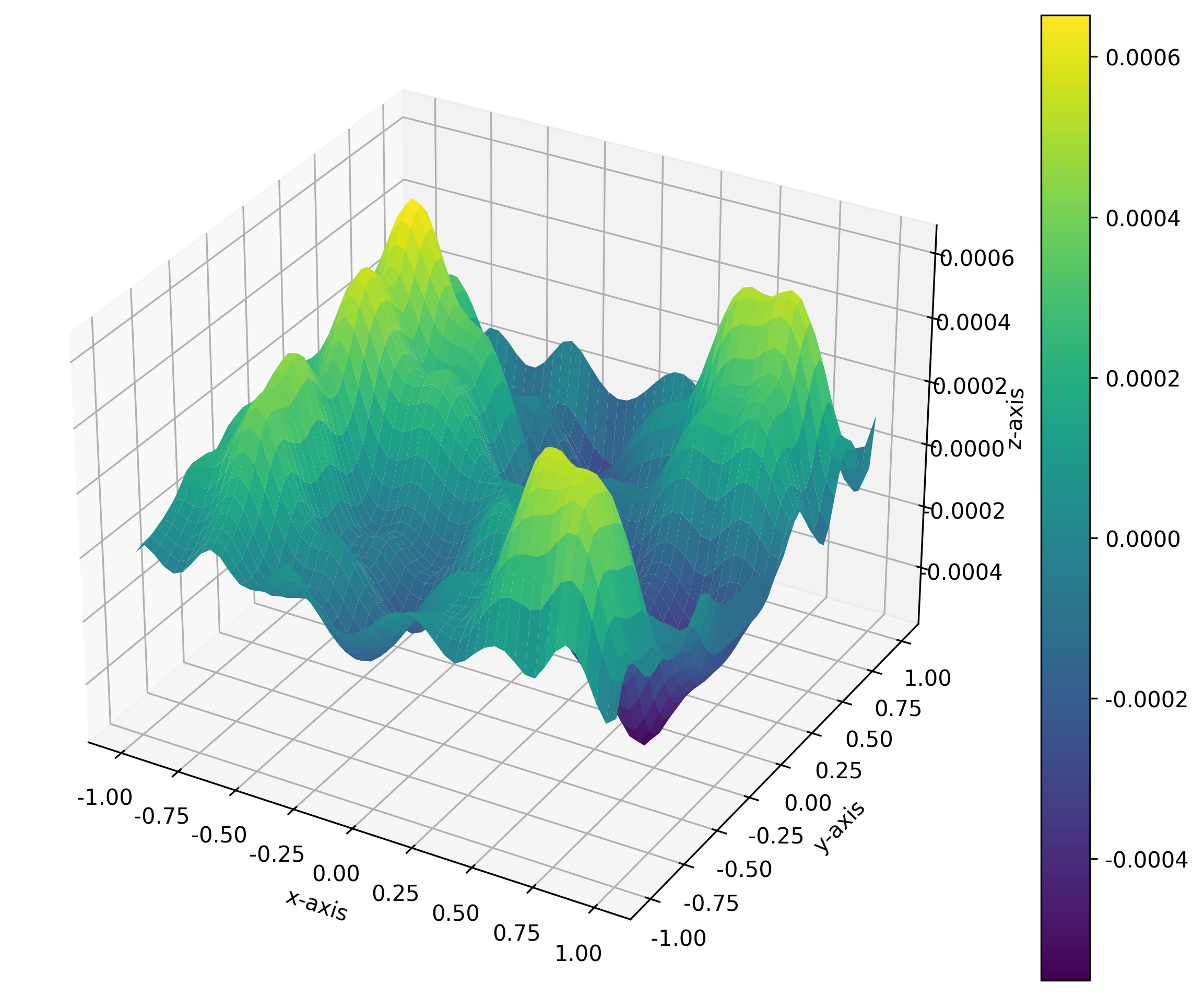}
  \caption{$\rm MIM_b$}
\end{subfigure}\hfill
\begin{subfigure}[t]{0.32\textwidth}
  \centering
  \includegraphics[width=\linewidth]{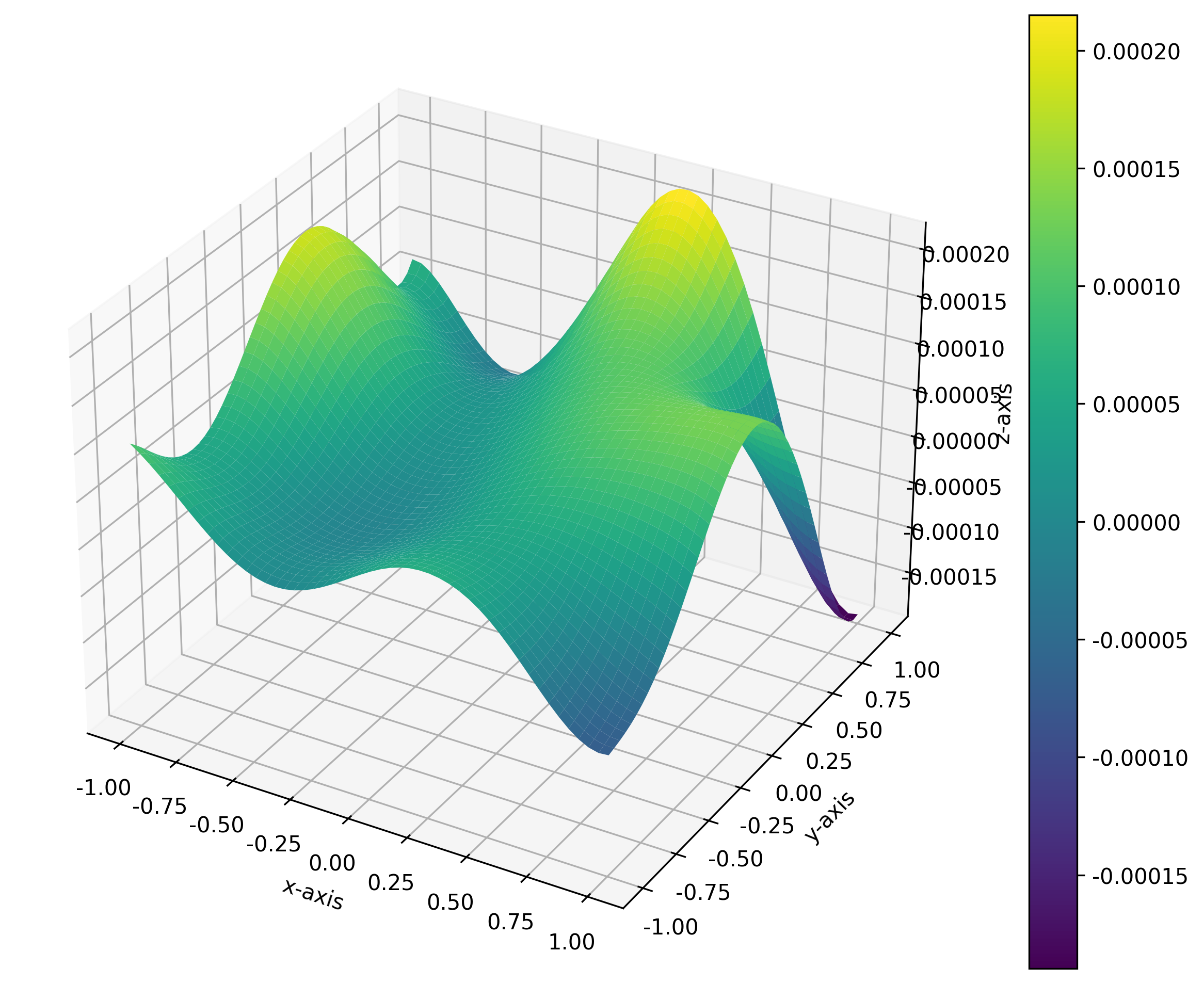}
  \caption{PINN}
\end{subfigure}

\caption{Pointwise error slice plots corresponding to the numerical solutions shown in Figure~\ref{fig:slice-solutions} on the plane $x_3 = 0.5$.}
\label{fig:slice-error}
\end{figure}

\begin{figure}[htbp]
\centering

\begin{subfigure}[t]{0.5\textwidth}
  \centering
  \includegraphics[width=\linewidth]{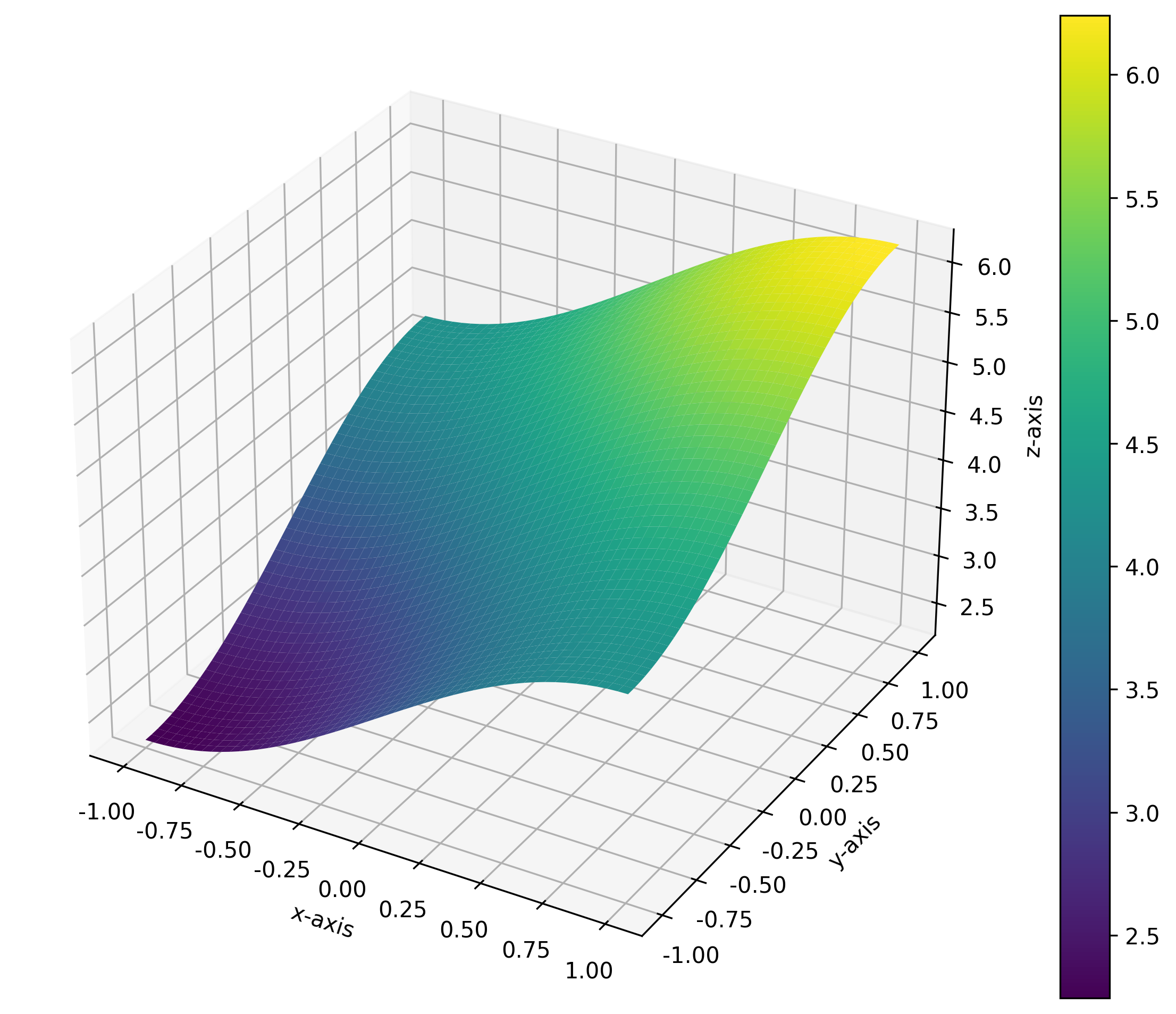}
  \caption{Exact}
\end{subfigure}

\vspace{0.5em}

\begin{subfigure}[t]{0.32\textwidth}
  \centering
  \includegraphics[width=\linewidth]{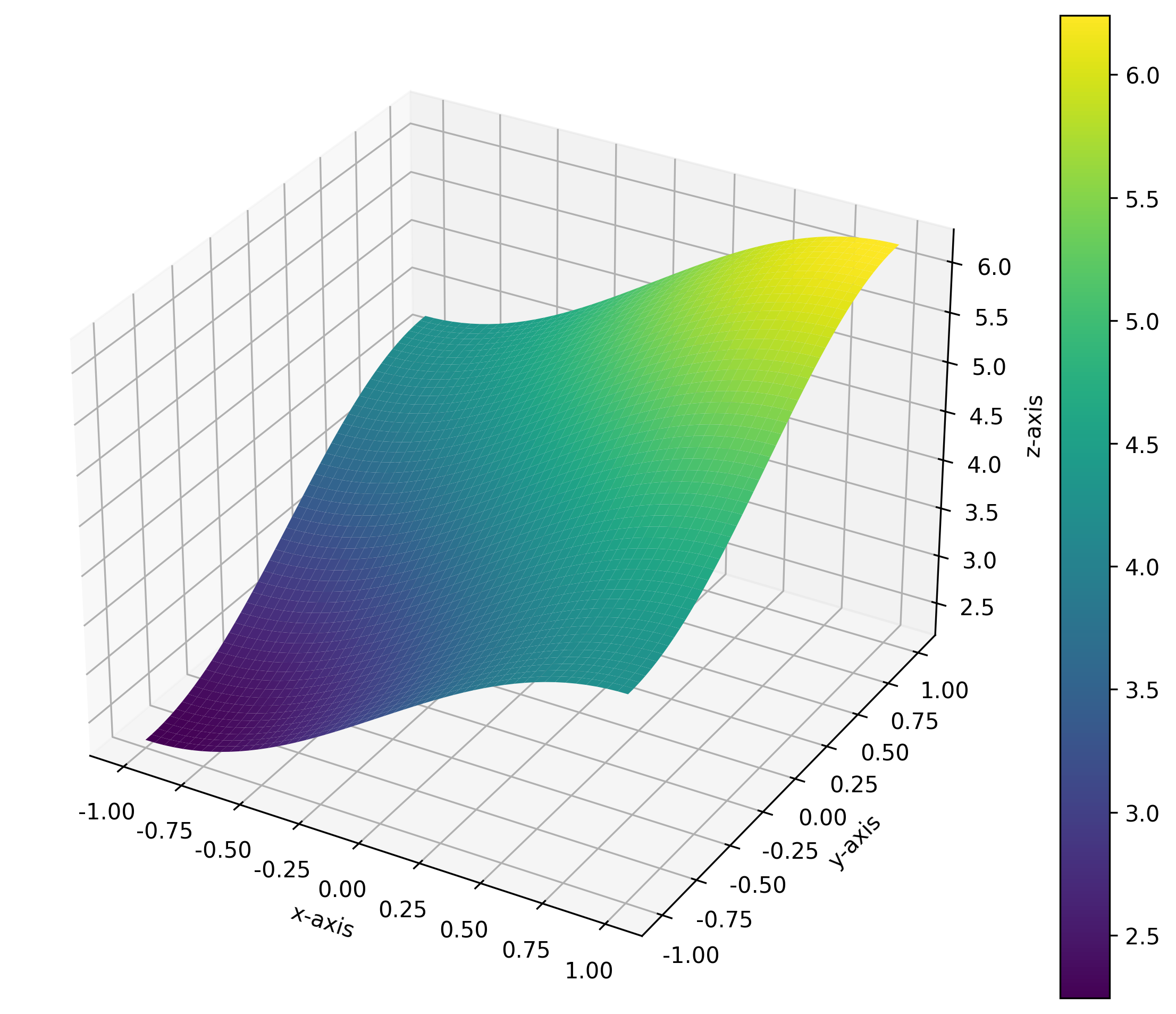}
  \caption{$\rm MIM_a$}
\end{subfigure}\hfill
\begin{subfigure}[t]{0.32\textwidth}
  \centering
  \includegraphics[width=\linewidth]{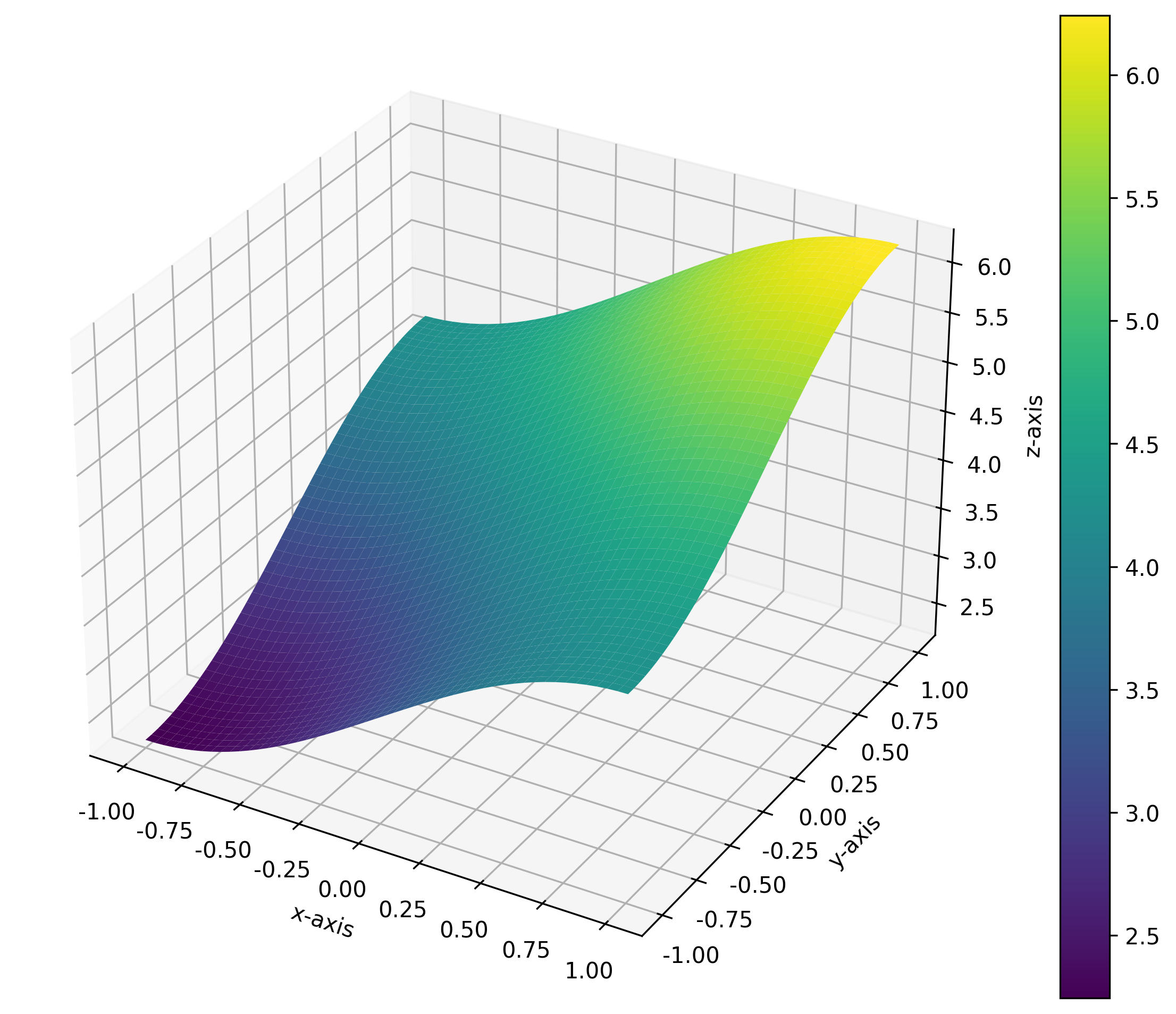}
  \caption{$\rm MIM_b$}
\end{subfigure}\hfill
\begin{subfigure}[t]{0.32\textwidth}
  \centering
  \includegraphics[width=\linewidth]{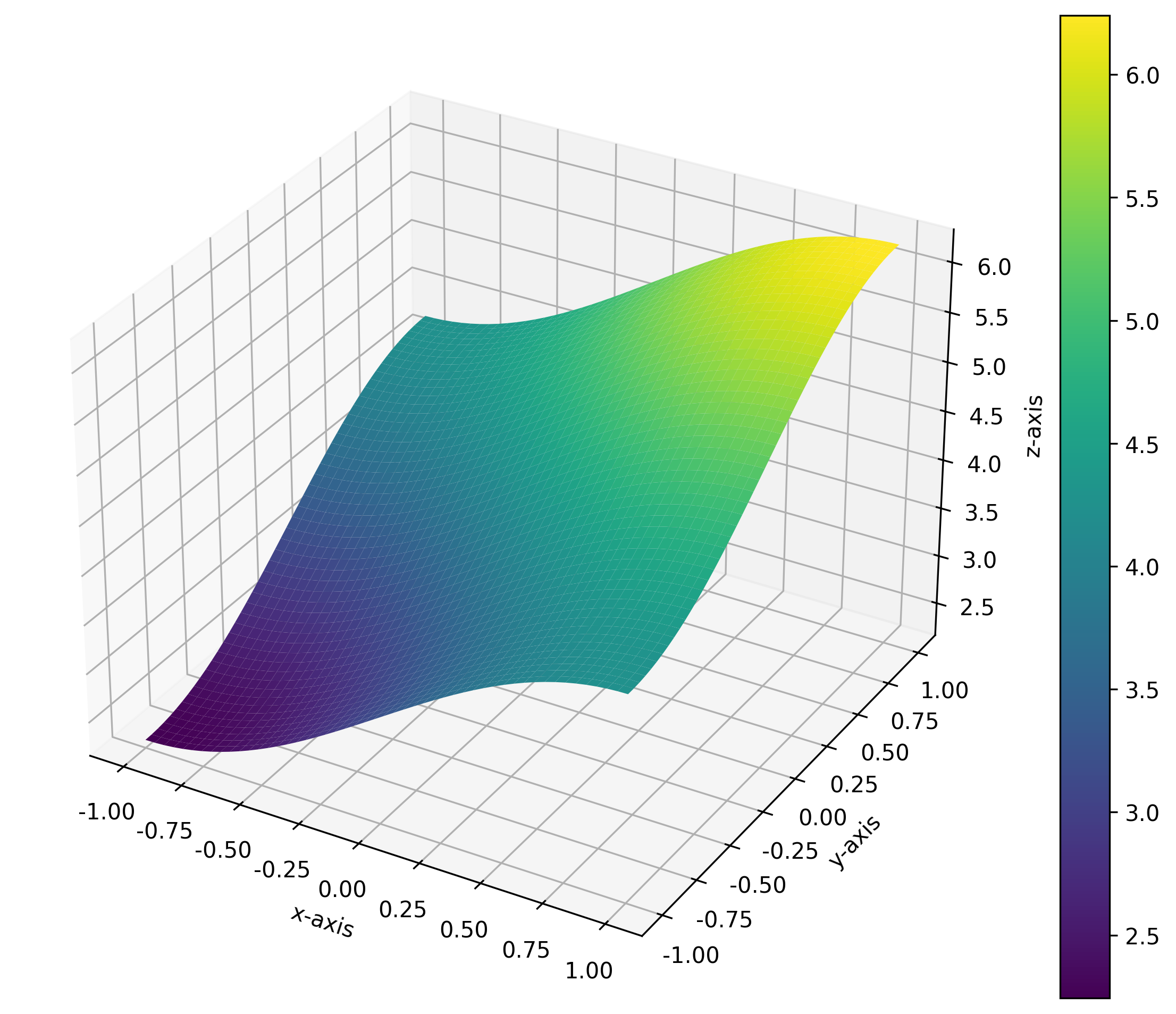}
  \caption{PINN}
\end{subfigure}

\caption{Slice plots of the exact solution and numerical solutions obtained by three methods on the plane $x_3 = \cdots = x_8 = 0.5$.}
\label{fig:slice-solutions1}
\end{figure}

\begin{figure}[htbp]
\centering

\begin{subfigure}[t]{0.32\textwidth}
  \centering
  \includegraphics[width=\linewidth]{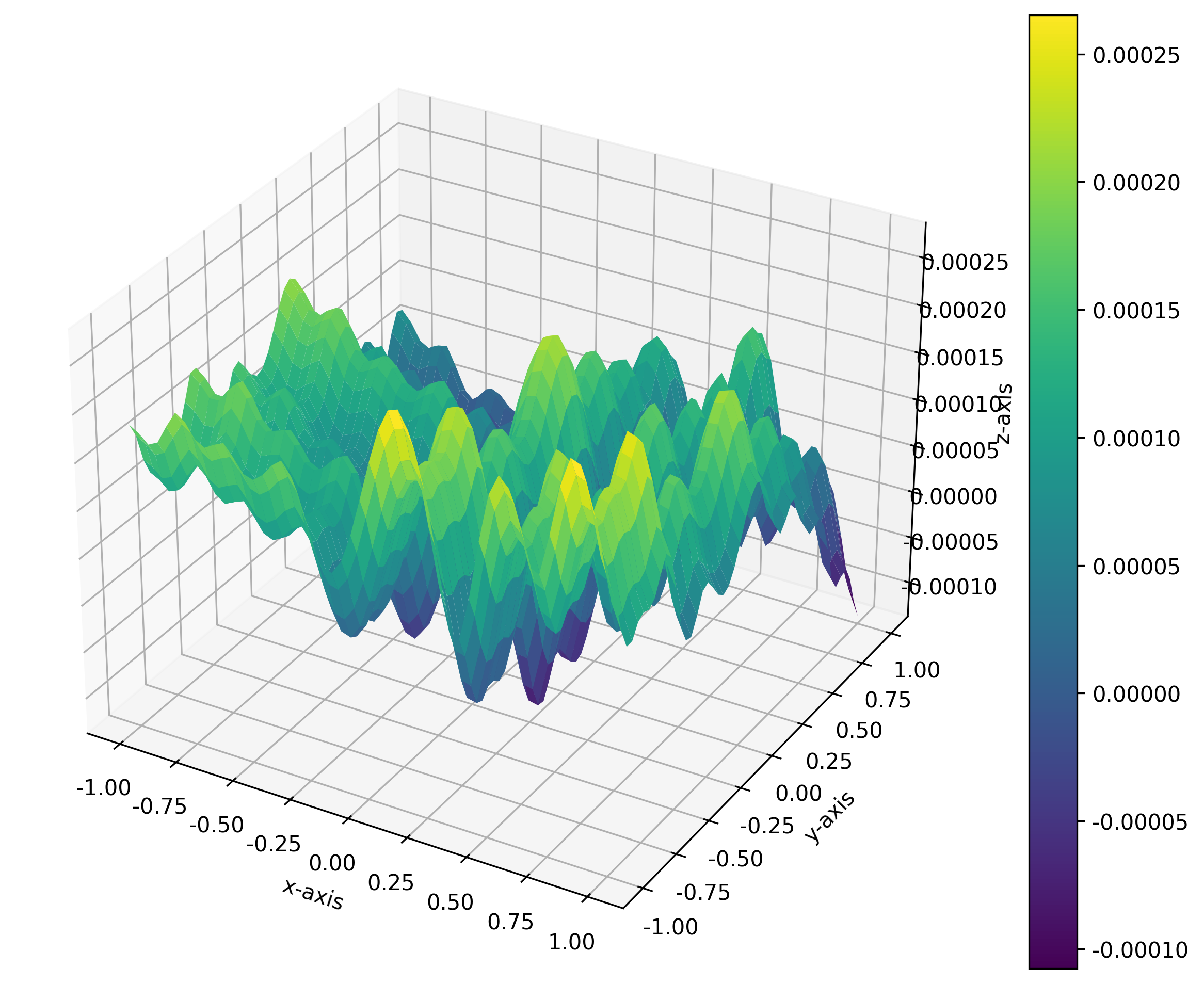}
  \caption{$\rm MIM_a$}
\end{subfigure}\hfill
\begin{subfigure}[t]{0.32\textwidth}
  \centering
  \includegraphics[width=\linewidth]{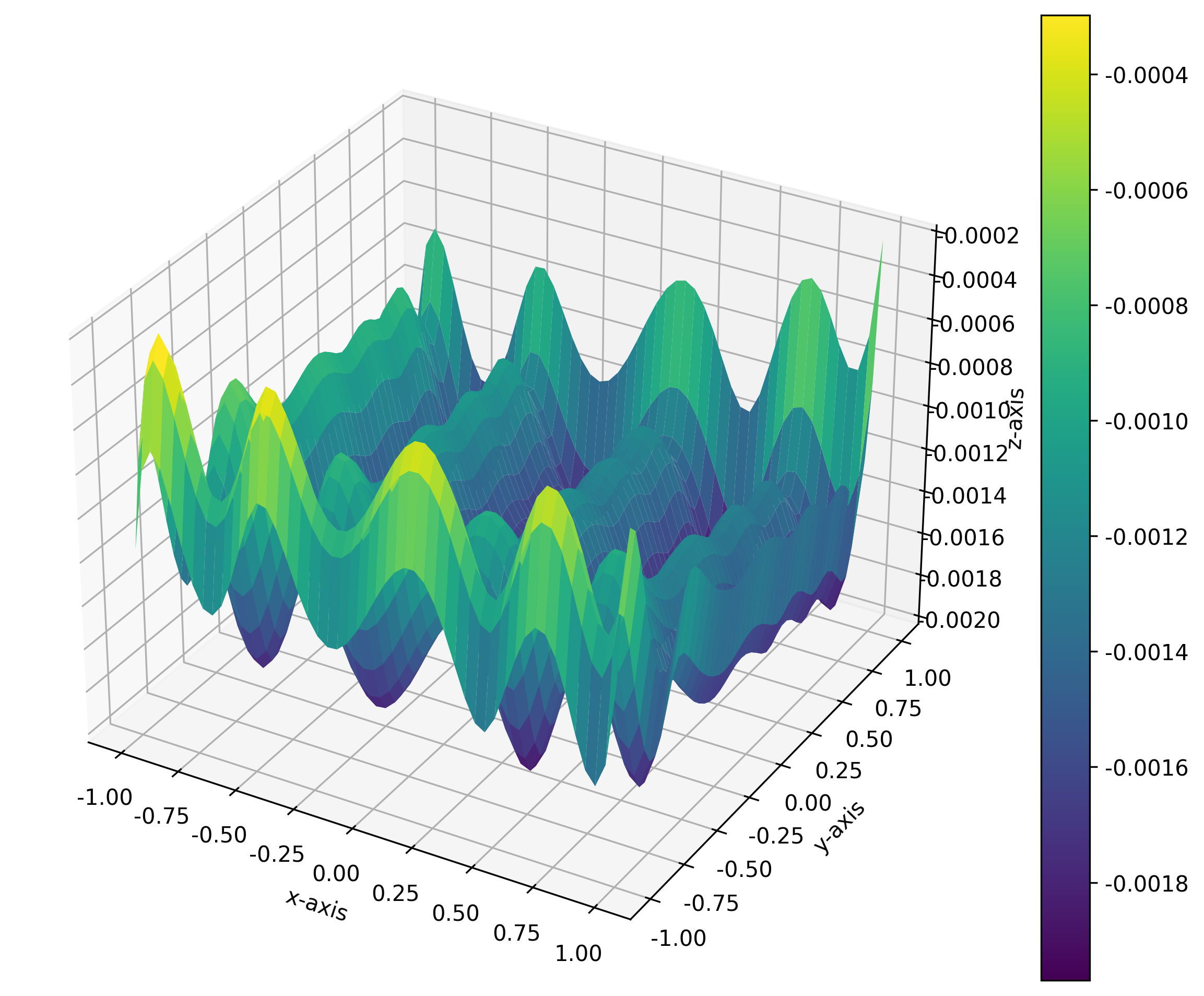}
  \caption{$\rm MIM_b$}
\end{subfigure}\hfill
\begin{subfigure}[t]{0.32\textwidth}
  \centering
  \includegraphics[width=\linewidth]{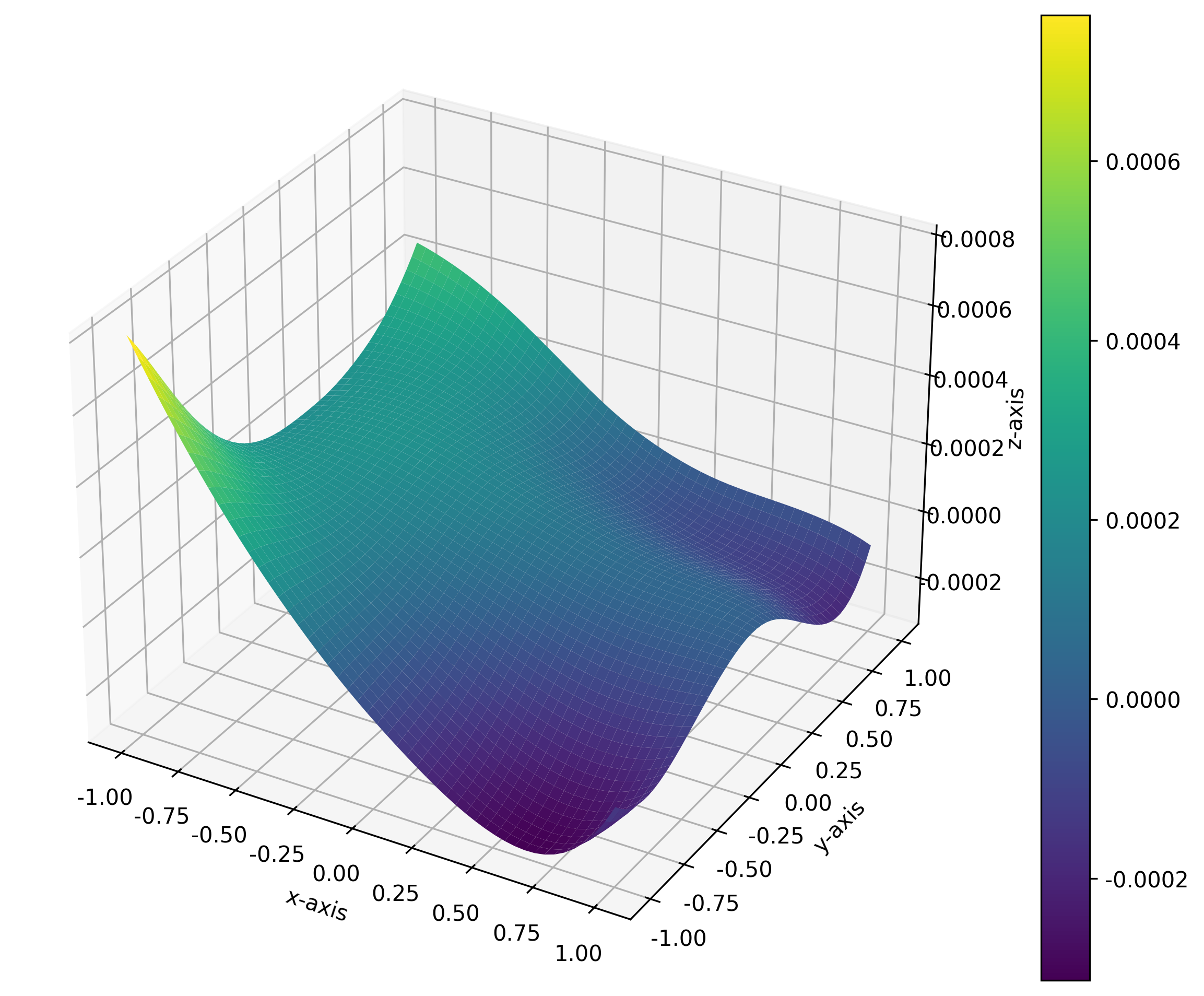}
  \caption{PINN}
\end{subfigure}

\caption{Pointwise error slice plots corresponding to the numerical solutions shown in Figure~\ref{fig:slice-solutions1} on the plane $x_3 =\cdots=x_8 = 0.5$.}
\label{fig:slice-error1}
\end{figure}

Finally, we present some numerical solutions in moderate and high dimensions. In particular, we show two-dimensional slice plots of the exact and numerical solutions in Figures \ref{fig:slice-solutions} and \ref{fig:slice-solutions1}. For the three-dimensional problem, the slice plots are taken on the plane $x_3 = 0.5$ (Figure \ref{fig:slice-solutions}). For the eight-dimensional case, 2D slices are obtained by varying $(x_1,x_2)$ while fixing all remaining variables at $0.5$ (Figure \ref{fig:slice-solutions1}). Furthermore, to clearly illustrate the approximation accuracy of the different methods, the corresponding pointwise error slice plots for both cases are shown in Figures \ref{fig:slice-error} and \ref{fig:slice-error1}, respectively. Notably, the PINN method yields smooth pointwise error in both cases.

In summary, the numerical experiments presented in this section are specifically designed to corroborate our theoretical findings. The consistent alignment between the empirical convergence rates and the theoretical error bounds (as established in Theorem \ref{thm2}) rigorously validates the proposed $\mathrm{MIM_a}$ framework. Furthermore, the stable performance across dimensions $d \in \{2,\,5,\, 6,\, 8\}$ confirms that the system splitting strategy effectively circumvents the curse of dimensionality, exactly as our theoretical analysis predicted.

 
 \section{Conclusion}
In this work, we proposed a mixed residual method and developed a rigorous error bound analysis for neural network approximations of the biharmonic equation subject to nonhomogeneous clamped boundary conditions. Under the assumption that the exact solution admits a spectral Barron space representation, we established a priori error estimates for shallow neural networks approximation that circumvent the curse of dimensionality. Furthermore, our numerical experiments corroborated these theoretical findings, demonstrating that the proposed method achieves an optimal balance in system splitting and outperforms standard comparative baselines. 

Although the same methodology is expected to extend to other higher-order elliptic equations, the treatment of boundary conditions in such problems is typically more delicate. We therefore restrict our current analysis to the classical clamped boundary condition. Several directions remain open for future investigation. For instance, various inverse problems for biharmonic equations are expected to be solvable by neural network-based methods, yet rigorous mathematical foundations for such approaches are still lacking.

\appendix
\section{Proof of Lemma \ref{lem_sigma2approximation}}
\label{se_appendix}
The proof proceeds along similar lines to the arguments developed in \cite{barron1993universal,xurefined} in order to establish a fast approximation rate for shallow neural networks with the $\mathrm{ReCU}$ activation function.

We begin with the main idea of the proof. For any function in the spectral Barron space, we first show that it belongs to the $H^3(\Omega)$-closure of the convex hull of a certain set. Estimating the metric entropy of this set and then applying \cite[Thm 1]{makovoz1996random} yields the fast approximation rate.

Recall that Taylor's expansion with integral remainder states that if $f : \mathbb{R} \to \mathbb{R}$ has $k+1$ continuous derivatives in some neighborhood $U$ of $x = a$, then for any $x \in U$,
\begin{equation*}
    f(x)
    =
    f(a) + f^{(1)}(a)(x-a) + \cdots + 
    \frac{f^{(k)}(a)}{k!}(x-a)^k + \int_0^x \frac{f^{(k+1)}(a)}{k!}(x-t)^k \mathrm{d}t.
\end{equation*}
For a function $f \in \mathcal{B}^4(\Omega)$, by the definition of the spectral Barron space we may assume that the infimum is attained at an extension $f_e$. To simplify notation, we write $f_e$ simply as $f$, since $f_e|_{\Omega} = f$. Using the Fourier inversion formula and the fact that $f$ is real-valued, we obtain
    \begin{equation*}
		\begin{aligned}
		f(x) = &\,
		\mathrm{Re}\int_{\mathbb{R}^d} e^{i\boldsymbol{\omega}\cdot\boldsymbol{x}}\hat{f}(\boldsymbol{\omega}) \mathrm{d}\boldsymbol{\omega}\\
		=&\,
		\mathrm{Re}\int_{\mathbb{R}^d} e^{i\boldsymbol{\omega}\cdot\boldsymbol{x}}e^{i\theta(\boldsymbol{\omega})}|\hat{f}(\boldsymbol{\omega})| \mathrm{d}\boldsymbol{\omega}
		=
	\int_{\mathbb{R}^d} \cos(\boldsymbol{\omega}\cdot\boldsymbol{x} + \theta(\boldsymbol{\omega}))|\hat{f}(\boldsymbol{\omega})| \mathrm{d}\boldsymbol{\omega}\\
	=&
	\int_{\mathbb{R}^d} \frac{B\cos(\boldsymbol{\omega}\cdot\boldsymbol{x} + \theta(\boldsymbol{\omega}))}{(1 + |\boldsymbol{\omega}|_1^2)^2} \Lambda(\mathrm{d}\boldsymbol{\omega})
	=
		\int_{\mathbb{R}^d}g(\boldsymbol{x},\boldsymbol{\omega}) \Lambda(\mathrm{d}\boldsymbol{\omega}),
		\end{aligned}
		\end{equation*}
where $B = \int_{\mathbb{R}^d}(1 + |\boldsymbol{\omega}|_1^2)^2|\hat{f}(\boldsymbol{\omega})|\mathrm{d}\omega$, and $\Lambda(\mathrm{d}\boldsymbol{\omega})
=\frac{1}{B}(1 + |\boldsymbol{\omega}|_1^2)^2|\hat{f}(\boldsymbol{\omega})|\mathrm{d}\boldsymbol{\omega}$ is a probability measure, $e^{i\theta(\boldsymbol{\omega})}$ denotes the phase of $\hat{f}(\boldsymbol{\omega})$, and
\begin{equation*}
	g(\boldsymbol{x},\boldsymbol{\omega})
	=
	\frac{B\cos(\boldsymbol{\omega}\cdot\boldsymbol{x} + \theta(\boldsymbol{\omega}))}{(1 + |\boldsymbol{\omega}|_1^2)^2}.
	\end{equation*}
We define the corresponding function class
    \begin{equation*}
	\mathcal{G}_{\cos}(B) 
	:=
	\Big\{
		\frac{B\cos(\boldsymbol{\omega}\cdot\boldsymbol{x} + \theta(\boldsymbol{\omega}))}{(1 + |\boldsymbol{\omega}|_1^2)^2} : \boldsymbol{\omega}\in \mathbb{R}^d, t\in\mathbb{R}
	\Big\}.
	\end{equation*}
Consequently, from the integral representation of $f$ and the form of $g$, we obtain the following lemma:

\begin{lemma}
Assume $f\in \mathcal{B}^4(\Omega)$ and let $B=\|f\|_{\mathcal{B}^4(\Omega)}$. Then $f$ belongs to the $H^3(\Omega)$-closure of the convex hull of the function class $\mathcal{G}_{\cos}(B)$.    
\end{lemma}
\begin{proof}
The proof proceeds via the probabilistic method. Assume that $\{\boldsymbol{\omega}_i\}_{i=1}^n$ is a sequence of i.i.d. random variables distributed according to $\Lambda$. Then
\begin{equation*}
    \begin{aligned}
        \mathbb{E}\Big[\big\| f(\boldsymbol{x}) - 
        \frac{1}{n}\sum_{i=1}^n g(\boldsymbol{x}, \boldsymbol{\omega}_i) \big\|^2_{H^3(\Omega)} \Big]
        =&
        \int_\Omega \sum_{|\alpha|\leq 3} \mathbb{E}
        \Big[ \big| D^\alpha f(\boldsymbol{x}) - \frac{1}{n}
        \sum_{i=1}^n D^\alpha g(\boldsymbol{x}, \boldsymbol{\omega}_i) \big|^2
        \Big] \, \mathrm{d}\boldsymbol{x} \\
        \leq&
        \frac{1}{n} \int_\Omega \sum_{|\alpha|\leq 3} \operatorname{Var}_{\boldsymbol{\omega}\sim \Lambda}\big( D^\alpha g(\boldsymbol{x},\boldsymbol{\omega}) \big) \, \mathrm{d}\boldsymbol{x}
        \leq
        \frac{C B^2}{n},
    \end{aligned}
\end{equation*}
where we have used the bound
\begin{equation*}
    D^\alpha g(\boldsymbol{x}, \boldsymbol{\omega}_i)
    \leq
    \frac{B|\boldsymbol{\omega}|_1^{|\alpha|}}{(1 + |\boldsymbol{\omega}|_1^2)^2} 
    \leq
    \frac{B}{1 + |\boldsymbol{\omega}|_1} \leq B.
\end{equation*}
Then, for any given tolerance $\varepsilon > 0$, the Markov's inequality yields
\begin{equation*}
\begin{aligned}
   P\Big(\big\| f(\boldsymbol{x}) - 
        \frac{1}{n}\sum_{i=1}^n g(\boldsymbol{x}, \boldsymbol{\omega}_i) \big\|_{H^3(\Omega)} > \varepsilon\Big)
        \leq &
        \frac{1}{\varepsilon^2}
        \mathbb{E}\Big[\big\| f(\boldsymbol{x}) - 
        \frac{1}{n}\sum_{i=1}^n g(\boldsymbol{x}, \boldsymbol{\omega}_i) \big\|^2_{H^3(\Omega)} \Big]
        \leq 
        \frac{C B^2}{\varepsilon^2 n}.
\end{aligned}
\end{equation*}
By choosing $n$ sufficiently large so that $\frac{C B^2}{\varepsilon^2 n} < 1$, we obtain
\begin{equation*}
     P\Big(\big\| f(\boldsymbol{x}) -
        \frac{1}{n}\sum_{i=1}^n g(\boldsymbol{x}, \boldsymbol{\omega}_i) \big\|_{H^3(\Omega)} \leq \varepsilon\Big) > 0,
\end{equation*}
which implies that there exist realizations of the random variables $\{\boldsymbol{\omega}_i\}_{i=1}^n$ for which $\big\| f(\boldsymbol{x}) -
        \frac{1}{n}\sum_{i=1}^n g(\boldsymbol{x}, \boldsymbol{\omega}_i) \big\|_{H^3(\Omega)} \leq \varepsilon$.

\end{proof}

Note that any function $g(\boldsymbol{x},\boldsymbol{\omega})$ can be expressed as the composition of a one-dimensional function
$$g(z) = \frac{B\cos(|\boldsymbol{\omega}|_1 z+ t)}{(1 + |\boldsymbol{\omega}|_1^2)^2},
    \quad \mbox{with a linear function } 
    z = \frac{\boldsymbol{\omega}}{|\boldsymbol{\omega}|_1}\cdot\boldsymbol{x},$$ 
where the variable $z$ takes values in $[-1,1]$. Consequently, in order to prove that $f$ belongs to the $H^3(\Omega)$-closure of the convex hull of the function class $\mathcal{F}_{\sigma_3}(cB)\cup \mathcal{F}_{\sigma_3}(-cB)\cup\{0\}$, it suffices to show that $g$ lies in the $H^3(\Omega)$-closure of the convex hull of the function class $\mathcal{F}_{\sigma_3}^1(cB)\cup \mathcal{F}_{\sigma_3}^1(-cB)\cup\{0\}$, where
     	\begin{equation*}
		\mathcal{F}_{\sigma_3}(b)
		:=
		\{b\sigma_3(\boldsymbol{\omega}\cdot\boldsymbol{x} + t) : |\boldsymbol{\omega}|_1 = 1, t\in [-1,1]\},
		\end{equation*}
	and
		\begin{equation*}
		\mathcal{F}_{\sigma_3}^1(b)
		:=
		\{b\sigma_3(\epsilon z+ t) : \epsilon = +1 \mbox{ or } -1, t\in [-1,1]\},
	\end{equation*}
	for any constant $b\in \mathbb{R}$. 

Since
	\begin{equation*}
		g(z)
		=
	\frac{B\cos(|\boldsymbol{\omega}|_1 z+ t)}{(1 + |\boldsymbol{\omega}|_1^2)^2}
		=
		\frac{B(\cos(|\boldsymbol{\omega}|_1 z)\cos t - \sin(|\boldsymbol{\omega}|_1 z)\sin t) }{(1 + |\boldsymbol{\omega}|_1^2)^2}
		\end{equation*}
with $z\in[-1,1]$, applying the Taylor's expansion with integral remainder to $\cos(|\boldsymbol{\omega}|_1 z)$ and $\sin(|\boldsymbol{\omega}|_1 z)$ about the point $0$ yields
   \begin{equation*}
	\cos(|\boldsymbol{\omega}|_1 z)
	=
	1 - \frac{|\boldsymbol{\omega}|_1^2}{2}z^2 + \int_{0}^z |\boldsymbol{\omega}|_1^4\cos(|\boldsymbol{\omega}|_1 s)\frac{(z-s)^3}{6}\mathrm{d}s,
	\end{equation*}
and
\begin{equation*}
	\sin(|\boldsymbol{\omega}|_1 z)
	=
	|\boldsymbol{\omega}|_1 z - \frac{|\boldsymbol{\omega}|_1^3}{6}z^3 + \int_{0}^z |\boldsymbol{\omega}|_1^4\sin(|\boldsymbol{\omega}|_1 s)\frac{(z-s)^3}{6}\mathrm{d}s.
\end{equation*} 

Note that $z^3$, $z^2$, $z$, and $1$ can be expressed as combinations of the $\sigma_3$ function on $[-1,1]$. Specifically,
\begin{equation*}
	\begin{gathered}
		z^3 
		=
		 \sigma_3(z) - \sigma_3(-z),\quad 
		z 
		=
		 \frac{(z+1)^3 + (z-1)^3 - 2z^3}{6},\\
		z^2 
		=
		 \frac{1}{3}z + \frac{7}{9}z^3 + \frac{1}{9}(z-1)^3 - \frac{8}{9}(z - \frac{1}{2})^3,\ 
		1 
		=
		-3z - 3z^2 - z^3 + (z+1)^3.
		\end{gathered}
	\end{equation*}

Therefore, we need only prove that the integral remainders lie in the $H^3([-1,1])$-closure of the convex hull of the function class $\mathcal{F}_{\sigma_3}^1(cB)\cup \mathcal{F}_{\sigma_3}^1(-cB)\cup\{0\}$. In the sequel, the constant $c$ may vary from line to line, but it always denotes a generic constant; hence we keep the same notation. 
Due to the form of the integral remainder, we consider the general expression
$$h(z) = \int_0^z \varphi(s)(z-s)^3\,\mathrm{d}s,$$
with $\varphi \in C([-1,1])$. Using the identity $(z-s)^3 = (z-s)^3_+ - (-z+s)^3_+$, we obtain
  \begin{equation*}
  	\begin{aligned}
  		h(z) 
  		=
  		 \int_0^z \varphi(s)(z-s)^3_+ \mathrm{d}s
  		- \int_0^z \varphi(s)(-z+s)^3_+ \mathrm{d}s
  		=:
  		A_1 + A_2.
  		\end{aligned}
  	\end{equation*}
We now show that
\begin{equation*}
		A_1 + A_2
		=
		 \int_0^1 \varphi(s)(z-s)^3_+ \mathrm{d}s
		 +
		 \int_0^1 \varphi(-s)(-z-s)^3_+ \mathrm{d}s 
		 =:
		 B_1 +B_2.
	\end{equation*}
For $z \ge 0$, it is straightforward to verify that
\begin{equation}
	A_1
	=
	\int_0^z \varphi(s)(z-s)^3_+ \mathrm{d}s
	=
	\int_0^1 \varphi(s)(z-s)^3_+ \mathrm{d}s
	= B_1, \mbox{ and } A_2 = B_2 = 0, 
	\end{equation}
which yields $A_1 + A_2 = B_1 + B_2$.
For $z < 0$, we similarly have $A_1 = B_1 = 0$, and
\begin{equation}
	\begin{aligned}
	A_2
	=&
	- \int_0^z \varphi(s)(-z+s)^3_+ \mathrm{d}s
	=
	\int_z^0 \varphi(s)(-z+s)^3_+ \mathrm{d}s\\
	=&
	\int_z^0 \varphi(s)(-z+s)^3_+ \mathrm{d}s
	+
		\int_{-1}^z \varphi(s)(-z+s)^3_+ \mathrm{d}s\\
		=&
		\int_{-1}^0 \varphi(s)(-z+s)^3_+ \mathrm{d}s
		=
			\int_{0}^1 \varphi(-s)(-z-s)^3_+ \mathrm{d}s
			= B_2
	\end{aligned}
	\end{equation}
where the third equality follows from the fact that $\int_{-1}^z \varphi(s)(-z+s)^3_+ \,\mathrm{d}s = 0$.
	Hence we conclude that
	\begin{equation*}
			h(z) 
			=
			\int_0^1 \varphi(s)(z-s)^3_+ \mathrm{d}s
			+
			\int_0^1 \varphi(-s)(-z-s)^3_+ \mathrm{d}s
			:= h_1 + h_2.
		\end{equation*}
The next step is to prove that $h_1$ and $h_2$ each lie in the $H^3([-1,1])$-closure of the convex hull of $\mathcal{F}_{\sigma_3}^1(cB)\cup \mathcal{F}_{\sigma_3}^1(-cB)\cup\{0\}$.

Note that $h_1^{(1)}(z) = 3\int_0^1 \varphi(s)(z-s)^2_+ \,\mathrm{d}s$, $h_1^{(2)}(z) = 6\int_0^1 \varphi(s)(z-s)_+ \,\mathrm{d}s$, and $h_1^{(3)}(z) = 6\int_0^1 \varphi(s) \mathbf{1}_{\{z-s \geq 0\}} \,\mathrm{d}s$ almost everywhere, since $(z-s)_+$ is differentiable in $z$ for almost every $s$. Let $\{s_i\}_{i=1}^n$ be an i.i.d. sequence of random variables distributed uniformly on the interval $[0,1]$. Then, by the Fubini's theorem,
\begin{equation*}
	\begin{aligned}
		\mathbb{E}&\Big\| h_1(z)- \frac{1}{n}\sum_{i=1}^n\varphi(s)(z-s)^3_+ \Big\|_{H^3([-1,1])}^2\\
		=&
		\int_{-1}^1
		\mathbb{E}\Big[
		\sum_{|\alpha|\leq 3}
        \big|D^\alpha h_1(z)- \frac{1}{n}\sum_{i=1}^n \varphi(s_i)D^\alpha(z-s_i)^3_+ \big|
		\Big]
		\mathrm{d}z
		\leq
		\frac{C}{n},
		\end{aligned}
	\end{equation*}	
where the last inequality follows from the boundedness of $\varphi$. The same conclusion holds for $h_2(z)$. Consequently, we deduce that $h$ belongs to the $H^3([-1,1])$-closure of the convex hull of the function class $\mathcal{F}_{\sigma_3}^1(cB)\cup \mathcal{F}_{\sigma_3}^1(-cB)\cup\{0\}$. 
 
By applying the variable substitution, we conclude that for any $f\in \mathcal{B}^4(\Omega)$ and any $\epsilon>0$, there exists a shallow $\sigma_3$ neural network such that
	\begin{equation}
		\|f(x)- \sum_{i=1}^m a_i\sigma_3(\boldsymbol{\omega}_i\cdot\boldsymbol{x} + t_i)\|_{H^3(\Omega)}
		\leq
		\epsilon,
		\end{equation}	
where $|\boldsymbol{\omega}_i|_1 = 1$, $|t_i|\leq 1$, $\sum_{i=1}^m|a_i|\leq c\|f\|_{\mathcal{B}^4(\Omega)}$, and $c$ is a universal constant.

Denote $\mathcal{F}_3(1) := \{\sigma_3(\boldsymbol{\omega}\cdot\boldsymbol{x} + t) : |\boldsymbol{\omega}|_1=1,\; t\in [-1,1]\}$. For simplicity, we write $\mathcal{F}_3$ for $\mathcal{F}_3(1)$. The metric entropy of the function class $\mathcal{F}_3$ is defined as
\begin{equation}
	\begin{aligned}
\epsilon_n(\mathcal{F}_3) 
	:=
	\inf\{\epsilon: \mathcal{F}_{\sigma_3} &
	\mbox{ can be covered by at most $n$ sets}\\
	 & \mbox{ of  diameter $\leq \epsilon$ under the $H^3$ norm}\}.
	\end{aligned}
	\end{equation}


Hence it remains only to estimate the metric entropy of the function class $\mathcal{F}_3$ under the $H^3$ norm.
\begin{lemma}[Estimate of the metric entropy]
For any $n\in \mathbb{N}$,
\begin{equation*}
    \epsilon_{n}(\mathcal{F}_3) \leq c\, n^{-\frac{1}{3d}},
\end{equation*}
where $c$ is a generic constant.
\end{lemma}
\begin{proof}
Denote $h_k = \sigma_3 (\boldsymbol{\omega}_k \cdot \boldsymbol{x} + t_k)$ for $k=1,2$. We claim that
\begin{equation}\label{h1-h2 H3}
\begin{aligned}
    \|h_1 - h_2\|_{H^3(\Omega)} 
    \leq\,&
    c\big(|\boldsymbol{\omega}_1 - \boldsymbol{\omega}_2|_1^2 + |t_1 - t_2|^2\\
    &+ 
    \|I_{\{\boldsymbol{\omega}_1 \cdot \boldsymbol{x} + t_1 \geq 0\}} 
    - I_{\{\boldsymbol{\omega}_2 \cdot \boldsymbol{x} + t_2 \geq 0\}}\|_{L^2(\Omega)}^2\big).
    \end{aligned}
\end{equation}
Indeed, since $\sigma_3$ is $12$-Lipschitz on $[-2,2]$, we have
\begin{equation*}
\begin{aligned}
    \|h_1 - h_2\|_{L^2(\Omega)}^2
    =&
    \|\sigma_3 (\boldsymbol{\omega}_1 \cdot \boldsymbol{x} + t_1) 
    -
     \sigma_3 (\boldsymbol{\omega}_2 \cdot \boldsymbol{x} + t_2)\|_{L^2(\Omega)}^2\\
    \leq&
    144\|(\boldsymbol{\omega}_1 - \boldsymbol{\omega}_2)\cdot \boldsymbol{x} + (t_1 - t_2)\|_{L^2(\Omega)}^2\\
    \leq&
    288 (|\boldsymbol{\omega}_1 - \boldsymbol{\omega}_2|_1^2 + |t_1 - t_2|).
    \end{aligned}
\end{equation*}
For the first derivative term we have
\begin{equation*}
    \begin{aligned}
        \nabla(h_1 - h_2)
        =&
        3\boldsymbol{\omega}_1\sigma_2 (\boldsymbol{\omega}_1 \cdot \boldsymbol{x} + t_1) 
    -
     3\boldsymbol{\omega}_2\sigma_2 (\boldsymbol{\omega}_2 \cdot \boldsymbol{x} + t_2)\\
     =&
     3(\boldsymbol{\omega}_1 - \boldsymbol{\omega}_2)\sigma_2 (\boldsymbol{\omega}_1 \cdot \boldsymbol{x} + t_1) 
     +
     3\boldsymbol{\omega}_2\big(  \sigma_2 (\boldsymbol{\omega}_1 \cdot \boldsymbol{x} + t_1) 
     -
     \sigma_2 (\boldsymbol{\omega}_2 \cdot \boldsymbol{x} + t_2)
     \big),
    \end{aligned}
\end{equation*}
which yields
\begin{equation*}
    \begin{aligned}
        \|\nabla(h_1 - h_2)\|_{L^2(\Omega)}^2
    \leq&\,
   18\|(\boldsymbol{\omega}_1 - \boldsymbol{\omega}_2)\sigma_2 (\boldsymbol{\omega}_1 \cdot \boldsymbol{x} + t_1) \|_{L^2(\Omega)}^2\\
 &  +
   18\|\boldsymbol{\omega}_2\big(  \sigma_2 (\boldsymbol{\omega}_1 \cdot \boldsymbol{x} + t_1) 
     -
     \sigma_2 (\boldsymbol{\omega}_2 \cdot \boldsymbol{x} + t_2) \|_{L^2(\Omega)}^2\\
     \leq&
     288|\boldsymbol{\omega}_1 - \boldsymbol{\omega}_2|_1^2
     +
     576\big(|\boldsymbol{\omega}_1 - \boldsymbol{\omega}_2|_1^2 + |t_1 - t_2|^2
     \big).
    \end{aligned}
\end{equation*}
Here the last inequality uses the facts that $\sigma_2$ is $4$-Lipschitz on $[-2,2]$ and $|\boldsymbol{\omega}_2|\leq |\boldsymbol{\omega}_2|_1$.

For the second derivative term, we note that
\begin{equation*}
    \begin{aligned}
    \nabla^2(h_1 - h_2)
    =&
        6\sum_{i,j=1}^d \big(\omega_{1i}\omega_{1j}
        \sigma_1 (\boldsymbol{w}_1 \cdot \boldsymbol{x} + t_1)
        -
        \omega_{2i}\omega_{2j}
        \sigma_1 (\boldsymbol{w}_2 \cdot \boldsymbol{x} + t_2)\big)\\
        =&
        6\sum_{i,j=1}^d\Big(
        ((\omega_{1i} - \omega_{2i} )\omega_{1j} + \omega_{2i}(\omega_{1j} - \omega_{2j}))
        \sigma_1 (\boldsymbol{\omega}_1 \cdot \boldsymbol{x} + t_1)\\
        &+
        \omega_{2i}\omega_{2j}\big(  \sigma_1 (\boldsymbol{\omega} \cdot \boldsymbol{x} + t_1) 
     -
     \sigma_1 (\boldsymbol{\omega}_2 \cdot \boldsymbol{x} + t_2)
     \Big).
    \end{aligned}
\end{equation*}
Proceeding similarly to the estimate for the first derivative, we obtain
\begin{equation*}
    \begin{aligned}
        \|\nabla^2(h_1 - h_2)\|_{L^2(\Omega)}^2
        \leq
         648\big(|\boldsymbol{\omega}_1 - \boldsymbol{\omega}_2|_1^2 
         + |t_1 - t_2|^2
     \big).
    \end{aligned}
\end{equation*}
For the third derivative term we have
\begin{equation*}
    \begin{aligned}
        \nabla^3(h_1 - h_2)
        =&
        6\sum_{i,j,k=1}^d \omega_{1i}\omega_{1j}\omega_{1k}
        I_{\{\boldsymbol{\omega} \cdot \boldsymbol{x} + t_1 \geq 0\}}
        -
        \omega_{2i}\omega_{2j}\omega_{2k}
        I_{\boldsymbol{\omega}_2 \cdot \boldsymbol{x} + t_2 \geq 0\}}\\
        =&
        6\sum_{i,j,k=1}^d \Big(
        (\omega_{1i}\omega_{1j}\omega_{1k} - \omega_{2i}\omega_{2j}\omega_{2k})I_{\{\boldsymbol{\omega}_1 \cdot \boldsymbol{x} + t_1 \geq 0\}}\\
       & +
        \omega_{2i}\omega_{2j}\omega_{2k}(I_{\{\boldsymbol{\omega}_1 \cdot \boldsymbol{x} + t_1 \geq 0\}}
        - I_{\{\boldsymbol{\omega}_2 \cdot \boldsymbol{x} + t_2 \geq 0\}})\Big).
    \end{aligned}
\end{equation*}
Using the bounds $\sum_{i,j,k}|\omega_{1i}\omega_{1j}\omega_{1k}|^2 \leq |\boldsymbol{\omega}_1|_1^3 = 1$ and $\sum_{i,j,k}|\omega_{1i}\omega_{1j}\omega_{1k} - \omega_{2i}\omega_{2j}\omega_{2k}|^2 \leq 9|\boldsymbol{\omega}_1 - \boldsymbol{\omega}_2|_1^2$, we deduce that
\begin{equation*}
    \|\nabla^3(h_1 - h_2)\|_{L^2(\Omega)}^2
    \leq
    36\|I_{\{\boldsymbol{\omega}_1 \cdot \boldsymbol{x} + t_1 \geq 0\}}
        - I_{\{\boldsymbol{\omega}_2 \cdot \boldsymbol{x} + t_2 \geq 0\}}\|_{L^2}^2 + 324|\boldsymbol{\omega}_1 - \boldsymbol{\omega}_2|_1^2.
\end{equation*}
Combining the estimates for derivatives of all orders yields \eqref{h1-h2 H3}.

Therefore, following the same approach as in the proof of Proposition A.3 in \cite{xurefined}, we conclude the lemma.    
\end{proof}

\section*{Acknowledgements}
This work was supported by National Key Research and Development Programs of China (No. 2023YFA1009103), NSFC (No. 92570106) and Science and Technology Commission of Shanghai Municipality (No. 23JC1400501). 

\section*{Data Availability Statement}
Data will be made available on request.

\section*{Conflict of Interest}
The authors have no conflicts of interest to declare that are relevant to the content of this article.

\bibliographystyle{amsplain}
\bibliography{reference}

\end{document}